\documentclass[11pt,a4paper]{article}%
\usepackage[centertags]{amsmath}
\usepackage{amsfonts}
\usepackage{amssymb}
\usepackage{amsthm}
\usepackage{epsfig}
\usepackage{setspace}
\usepackage{ae}
\usepackage{eucal}
\usepackage[usenames]{color}%
\usepackage{graphicx}

\theoremstyle{plain}
\newtheorem{thm}{Theorem}[section]
\newtheorem{lem}[thm]{Lemma}
\newtheorem{cor}[thm]{Corollary}
\newtheorem{prop}[thm]{Proposition}
\theoremstyle{definition}

\theoremstyle{remark}

\newtheorem{rem}[thm]{Remark}

\renewcommand{\div}{\operatorname{div}}

\begin{document}

\title{ On the multivariate Burgers equation and the incompressible Navier-Stokes equation (part II)}
\author{J\"org Kampen }
\maketitle

\begin{abstract} 
We consider global schemes with $L^{\infty}\times H^2$-estimates for the multivariate Burgers equation and the incompressible Navier-Stokes equation in its Leray projection form. We extend the definition of a global scheme and the global $L^2$-theory of the multivariate Burgers equation to the incompressible Navier-Stokes equation, where we consider estimates of the gradient of the pressure in its Leray projection form in $H^1$ at each time step of the scheme. The extended scheme has a simplified control function (simplified compared to \cite{KNS}), which controls the growth of the Leray projection term. We use the fact that the data in our scheme are in $H^2$ at each time step  such that the data for the Poisson equation related to the elimination of pressure are in $H^1$ at each approximation step. The  growth of the velocity components is linear with respect to the $L^{\infty}\times H^2$-norm on a time scale of time step size $\rho_l\sim \frac{1}{l}$ at each time step $l\geq 1$. The time-step size is 'minimal' in order to make the scheme global while at the same time it provides us with uniform bounds for the first order coefficients of linear approximations of the local solutions. The difference to the recently revised part III of this scheme uses fundamental solutions of scalar parabolic equations with variable first order terms at each time step. The related representations of solutions are not convolutions. Therefore we use the adjoints of the fundamental solutions at each time step. A second difference is that we do not use inheritence of polynomial decay of higher order at each time step. The deacy at spatia infinity is controlled by the Sobolev norm (and by the simplified control function).  
\end{abstract}


2000 Mathematics Subject Classification. 35K40, 35Q30.
\section{Introduction}
In \cite{KB1}  we considered a constructive global scheme for the multivariate Burgers
equation Cauchy problem
\begin{equation}\label{qparasyst1}
\left\lbrace \begin{array}{ll}
\frac{\partial u_i}{\partial t}=\nu\sum_{j=1}^n \frac{\partial^2 u_i}{\partial x_j^2} 
-\sum_{j=1}^n u_j\frac{\partial u_i}{\partial x_j},\\
\\
\mathbf{u}(0,.)=\mathbf{h},
\end{array}\right.
\end{equation}
on a domain $[0,\infty)\times {\mathbb R}^n$, and where the viscosity constant $\nu$ is strictly positive, i.e., $\nu>0$. In the inviscous case $\nu=0$ singlularities may appear at leats for a consireable set of initial data, and it is not possible to define a global scheme in general. In addition to $\nu >0$ we assumed $h_i\in H^s\equiv H^s\left( {\mathbb R}^n\right)$ for $s\in {\mathbb R}$ and do this henceforth in this paper. This assumption is rather strong. It is satisfied if in addition to smoothness of the data $h_i\in C^{\infty}\left({\mathbb R}^n\right),\leq i\leq n$ we have polynomial decay at infinity for the function itself and its partial derivatives of arbitrary order. Here we say that a function $g\in C^{\infty}\left( {\mathbb R}^n\right) $ has polynomial decay of order $m>0$ up to derivatives of order $p>0$ at infinity if for all multiindices $\alpha=(\alpha_1,\cdots,\alpha_n)$ with order $|\alpha|:=\sum_{i=1}^n\alpha_i\leq p$ we have
\begin{equation}
|D^{\alpha}_xg(x)|\leq \frac{C_{\alpha}}{1+|x|^m}
\end{equation}
for some finite constants $C_{\alpha}$. Depending on the order of polynomial decay which we require the latter condition may become a little bit stronger than the former. For example in case of diemsnion $n=3$ the function
\begin{equation}
x\rightarrow \frac{1}{1+|x|^2}
\end{equation}
is smooth and has square integrable multivariate derivatives, i.e. is in $H^m\left({\mathbb R}^3\right)$ for integers $m\geq 0$, but it has only polynomial decay of second order at infinity. Anyway the latter classical condition is satisfied by all solution of physical interest and may also be assumed. 
Furthermore, in order to have a uniform representation of the following argument we assume that $n\geq 3$. Some arguments require indeed that $n=3$, which is another natural assumption, but the additional problems of our approach for $n>3$ are not insurmountable. Therefore, we shall indicate where we need $n=3$.
The differences of the multivariate Burgers problem to the Cauchy problem for the incompressible Navier-Stokes equation, i.e.,
\begin{equation}\label{nav}
\left\lbrace\begin{array}{ll}
    \frac{\partial\mathbf{v}}{\partial t}-\nu \Delta \mathbf{v}+ (\mathbf{v} \cdot \nabla) \mathbf{v} = - \nabla p, \\
    \\
\nabla \cdot \mathbf{v} = 0,~~~~t\geq 0,~~x\in{\mathbb R}^n,\\
\\
\mathbf{v}(0,.)=\mathbf{h},
\end{array}\right.
\end{equation}
are the source term in form of the negative gradient of the pressure, and the incompressibility condition. A classical solution with velocity components $v_i,~1 \leq i\leq n$, and such that $v_i(t,.)\in C^2\cap H^2$ for all $t$ solves 
\begin{equation}\label{Navleray}
\left\lbrace \begin{array}{ll}
\frac{\partial v_i}{\partial t}-\nu\sum_{j=1}^n \frac{\partial^2 v_i}{\partial x_j^2} 
+\sum_{j=1}^n v_j\frac{\partial v_i}{\partial x_j}=\\
\\ \hspace{1cm}\int_{{\mathbb R}^n}\left( \frac{\partial}{\partial x_i}K_n(x-y)\right) \sum_{j,k=1}^n\left( \frac{\partial v_k}{\partial x_j}\frac{\partial v_j}{\partial x_k}\right) (t,y)dy,\\
\\
\mathbf{v}(0,.)=\mathbf{h},
\end{array}\right.
\end{equation}
for $1\leq i\leq n$, and vice versa. In this case the Poisson equation
\begin{equation}\label{poissonp}
\begin{array}{ll} 
- \Delta p=\sum_{j,k=1}^n \left( \frac{\partial}{\partial x_k}v_j\right) \left( \frac{\partial }{\partial x_j}v_k\right)
\end{array}
\end{equation}
has the well-defined solution
\begin{equation}\label{pp}
p(t,x)=-\int_{{\mathbb R}^n}K_n(x-y)\sum_{j,k=1}^n\left( \frac{\partial v_k}{\partial x_j}\frac{\partial v_j}{\partial x_k}\right) (t,y)dy,
\end{equation}
where 
\begin{equation}
K_n(x):=\left\lbrace \begin{array}{ll}\frac{1}{2\pi}\ln |x|,~~\mbox{if}~~n=2,\\
\\ \frac{1}{(2-n)\omega_n}|x|^{2-n},~~\mbox{if}~~n\geq 3\end{array}\right.
\end{equation}
is the Poisson kernel. Here, we mention the case $n=2$ and invite the reader to apply the following argument in the case $n=2$. However, since this case is different from all cases $n\geq 3$ we assumed $n\geq 3$ above and stay with this assumption. We also mention that $|.|$ denotes the Euclidean norm and $\omega_n$ denotes the area of the unit $n$-sphere. In equation (\ref{Navleray}) we use the fact that the formula (\ref{pp}) has a well defined gradient such that
\begin{equation}\label{pp1}
-\nabla p(t,x)=\int_{{\mathbb R}^n}\nabla K_n(x-y)\sum_{j,k=1}^n\left( \frac{\partial v_k}{\partial x_j}\frac{\partial v_j}{\partial x_k}\right) (t,y)dy.
\end{equation}
This is the term we need to control in order to extend the scheme for the multivariate Burgers equation in \cite{KB1} to the present situation.
 Note that for $n= 3$ the functions
\begin{equation}\label{Kkernel}
x\rightarrow \frac{\partial}{\partial x_l}K_n(x)=\omega_n^{-1}\frac{x_l}{|x|^n}
\end{equation}
are not integrable outside a ball $B^n_R(0):=\left\lbrace x||x|\leq R\right\rbrace $ with origin $0$ and radius $R>0$, i.e., are not in  $L^1\left( {\mathbb R}^n\setminus B^n_R(0)\right) $. It is useful to note that we have local integrability in the $L^1$-sense of the first partial derivatives of the kernel $K$.   For $n\geq 3$ we also have square integrability of these first order partial derivatives outside a ball around zero since
\begin{equation}
\int_{{\mathbb R}^n\setminus B^n_R(0)}\omega_n^{-1}\frac{|x_l|^2}{|x|^{2n}}dx\precsim \omega_n^{-1}\int_{r\geq R}\frac{dr}{r^{2n-2-n+1}}<\infty.
\end{equation}
Integrability in the $L^1$-sense of functionals as in (\ref{pp1}) is harder to achieve, even if we invoke partial derivatives using some information about the gradient of velocity. Note that even if we consider the second derivative of the kernel, then we get
\begin{equation}\label{pk2}
\frac{\partial^2}{\partial x_i\partial x_j}K_n(x):=\left\lbrace \begin{array}{ll} -\frac{n x_ix_j}{\omega_n}|x|^{-n-2},~~\mbox{if}~~i\neq j,\\
\\ \frac{|x|^2-nx^2_j}{\omega_n}|x|^{-n-2},~~\mbox{if}~~i=j,\end{array}\right.
\end{equation}
and these functions are surely square integrable outside a ball around the origin for $n\geq 3$ but not intergable in the $L^1$ sense outside a ball around zero. Now the 
gradient of the pressure in (\ref{pp}) contains a source of a Poisson equation in the form
\begin{equation}\label{source}
 \sum_{j,k=1}^n\left( \frac{\partial v_k}{\partial x_j}\frac{\partial v_j}{\partial x_k}\right).
\end{equation}
If our approximations $v^{\rho,m,l}_i$ (cf. scheme below) of the velocity components  $v_k$ in (\ref{source}) at each time step $l$ are in $H^2$ and have a uniform bound for all iteration step numbers $m$ of the approximation such that we have a bound $C+lC$ that is linear with respect to the time step number $l$, then we have the possibility to extend $L^2$-theory for constructive scheme of the multivariate Burgers equation \cite{KB1} to the incompressible Navier-Stokes equation. It turns out that this is difficult to achieve for the original scheme, but for an equivalent controlded scheme it is possible. We considered controlled scheme in \cite{KNS} and in \cite{K3}. However, the control function considered in this paper is considerably simpler. Indeed, if we could find an upper bound for the Lera pro jection term which grows linearly wth respect to the time step number, then it would be quite easy to set up a global scheme. In this case we may use for all approximations the pointwise relation
\begin{equation}\label{source2}
\begin{array}{ll}
 \sum_{j,k=1}^n\left( \frac{\partial v^{\rho,m,l}_k}{\partial x_j}\frac{\partial v^{\rho,m,l}_j}{\partial x_k}\right)(\tau,x)\\
 \\
 \leq \sum_{j,k=1}^n\frac{1}{2}\left( \left( \frac{\partial v^{\rho,m,l}_k}{\partial x_j}\right)^2(\tau,x)+ \left( \frac{\partial v^{\rho,m,l}_j}{\partial x_k}\right)^2\right)(\tau,x)
 \end{array}
\end{equation}
and the linear bound leads to the estimates
\begin{equation}\label{psource}
\begin{array}{ll}
{\Big |}\sum_{j,k=1}^n\left( \frac{\partial v^{\rho,m,l}_k}{\partial x_j}\frac{\partial v^{\rho,m,l}_j}{\partial x_k}\right){\Big |}_{L^1}
\leq \sum_{j,k=1}^n\frac{1}{2}\left( {\Big |}v^{\rho,m,l}_k{\Big |}_{H^1}+ {\Big |} v^{\rho,m,l}_j{\Big |}_{H^1}\right) \\
\\
=n\sum_{k=1}^n{\Big |}v^{\rho,m,l}_k{\Big |}_{H^1}\leq n^2(C+lC)
\end{array}
\end{equation}
As we said, it seems diffcult to get such a linear bound directly, but for an equivalent controlled function
\begin{equation}
v^{r,\rho,l}_i=v^{\rho,l}_i+r^l_i,~1\leq i\leq n
\end{equation}
with some simple bounded differentiable functions $r^l_i,~1\leq i\leq n$ defined on $[l-1,l]\times {\mathbb R}^n$ such an estimate can be achieved. We shall observe that we can define $r^l_i$ even in such a way that the functions $v^{r,\rho}_i:[0,\infty)\times {\mathbb R}^n\rightarrow {\mathbb R}$ which equal the restrictions of $v^{r,\rho,l}_i$ on $[l-1,l]\times {\mathbb R}^n$ have classical time derivatives across the discrete time values $\tau=l$ for all nonnegative integers $l\geq 0$. Our scheme uses classical representations in terms of fundamental solutions of linear parabolic approximating equations. These representations are not covolutions but the involved fundamental solutions and their first order spatial derivatives have Gaussian majorants and these can be used in order to apply estimates for convolutions.   
For the estimation of convolutions  we use the generalized Young inequality, i.e., the fact that for $1\leq p,q,r\leq \infty$ 
\begin{equation}\label{Y1}
 f\in L^q~\mbox {and}~ g\in L^p~\rightarrow  f\ast g\in L^r,\mbox{if}~\frac{1}{p}+\frac{1}{q}=1+\frac{1}{r},
\end{equation}
and
\begin{equation}\label{Y2}
|f\ast g|_{L^r}\leq |f|_{L^p} |g|_{L^q}
\end{equation}
for  a convolution $f\ast g$. These convolution estimates will be useful to us in the form
\begin{equation}\label{Y3}
|f\ast g|_{L^2}\leq |f|_{L^2} |g|_{L^1},
\end{equation}
and in the form
\begin{equation}\label{Y4}
\begin{array}{ll}
|f\ast g|_{L^{\infty}}\leq |f|_{L^p} |g|_{L^q}\mbox{ for }\\
\\
p=q=\frac{1}{2}\mbox{ and }r=1\mbox{ and }p=\infty,~q=\frac{1}{2}~\mbox{and}~r=1
\end{array}
\end{equation}
For example for a localized Laplacian kernel (\ref{psource}) together with (\ref{Y1}) and (\ref{Y2}) for $r=2$ and $p=1$ (estimate of the source), and $q=2$ ($L^2$-estimate of the source outside a ball around zero) lead us close to an $L^2$-bound for the approximations of the gradient of the pressure. In order to show that our scheme is global next to a simple control function and a specific time scale we use properties of the convolutions (shifting derivatives) and the fact that the velocity approximations are in $H^2$. We also use Gaussian a priori estimates of fundamental solutions and their first order derivatives.  These Gaussian upper bounds lead to upper bound convolutions in the representation of approximations $v^{\rho,m,l}_j$ of the velocity, which is convenient in order to apply the generalized Young inequality.
In order to achieve estimates in different function spaces we can extend this idea observing that the part of the source term considered in (\ref{source}) contains derivatives which may be shifted to the kernel (by some elementary estimates and using partial integration) outside a ball of origin zero while inside this ball we may use the local regularity of the function in (\ref{pp})- or its counterparts in the scheme. However, for $L^2$-theory it is enough to observe that for some smooth function $\phi$ supported on the ball of radius $1$, i.e., a function $\phi\in C^{\infty}(B_1(0))$ such that we have $\phi K_{,i}\in L^1$ and $(1-\phi)K_{,i}\in L^2$. We shall consider the application to other function spaces elsewhere. Now in a scheme we may consider local linearizations in which an approximation function $v^{f,\rho,l}_j,~1\leq j\leq n$ follows (with respect to an iteration) an approximative function $v^{g,\rho,l}_j,~1\leq j\leq n$ on a domain $D^{\tau}_l:=[l-1,l]\times {\mathbb R}^n$ with step size $\rho_l\sim \frac{1}{l}$ and transformed time $\tau =\rho_l\tau$ (similar as in \cite{KB1}). We then have to find a linear bound on the growth of the velocity function components, i.e., a bound for differences of form (using Einstein notation for usual derivatives)
\begin{equation}\label{burgfg*}
\begin{array}{ll}
v^{f,\rho,l}_i(\tau,x)-v^{g,\rho,l}_i(\tau,x)=\\
\\
=-\rho_l\int_{(l-1)}^{\tau}\int_{{\mathbb R}^n}\sum_{j=1}^n\left( f_j-g_j\right)(s,y)\frac{\partial v^{g,\rho,l}_i}{\partial x_j}(s,y)\Gamma^l_f(\tau,x;s,y)dyds+\\
\\
\int_{l-1}^{\tau}\rho_l\int_{{\mathbb R}^n}\int_{{\mathbb R}^n}K_{n,i}(z-y)\times\\
\\
{\Big (} \left( \sum_{j,k=1}^n\left( f_{k,j}+g_{k,j}\right)(s,y) \right) \times\\
\\
\left( \left( f_{j,k}-g_{j,k}\right)(s,y) \right) {\Big)}\Gamma^l_f(\tau,x;s,z)dydzds,\\
\end{array}
\end{equation}
and where $\Gamma^l_f$ is the fundamental solution of the linear (scalar!) parabolic equation
\begin{equation}
\frac{\partial \Gamma^l_f}{\partial \tau}-\rho_l\nu \Delta \Gamma^l_f+\rho_l\sum_{j=1}^n f_j\frac{\partial \Gamma^l_f}{\partial x_j}=0.
\end{equation}
For $n=3$ the natural bound is with respect to a $H^2$-norm, and once this is known it is easy to see that we can get a bound with respect to the $H^m$-norm for $m$ arbitrary as well. A linear bound is sufficient by the choice of step size $\rho_l\sim \frac{1}{l}$ if we have a gobal linear bound of the Leray projection term. It seems that at this point it is necessary to introduce a control function. The idea of a simplified control function is as follows. We introduce for $1\leq i\leq n$ functions $r_i:\left[0,\infty\right)\times {\mathbb R}^n\rightarrow {\mathbb R}$ which are linearly bounded with respect to time with bounded first order time derivatives. Furthermore the functions $r_i$ are bounded with respect to the spatial variables and have bounded spatial derivatives up to second order (a function space which we may denote by the symbol $R$). Next consider the function  
\begin{equation}
\mathbf{v}^{r}:=\mathbf{v}+\mathbf{r},
\end{equation}
where $\mathbf{v}=\left(v_1,\cdots v_n\right)^T$ and  $\mathbf{r}=\left(r_1,\cdots r_n\right)^T$.  These control functions $r_i$ are not known a priori, of course, but are constructed inductively with the time step number $l\geq 1$. It is shown then within the analysis of the scheme that they are indeed globally linearly bounded with respect to time and satisfy some other convenient properties. 
Note that the function $\mathbf{v}=\left(v_1+r_1,\cdots v_n+r_n\right)^T$
satisfies the equation   
\begin{equation}\label{Navleray}
\left\lbrace \begin{array}{ll}
\frac{\partial v^r_i}{\partial t}-\nu\sum_{j=1}^n \frac{\partial^2 v^r_i}{\partial x_j^2} 
+\sum_{j=1}^n v^r_j\frac{\partial v^r_i}{\partial x_j}=\\
\\
\frac{\partial r_i}{\partial t}-\nu\sum_{j=1}^n \frac{\partial^2 r_i}{\partial x_j^2} 
+\sum_{j=1}^n r_j\frac{\partial v^r_i}{\partial x_j}+\sum_{j=1}^n v^r_j\frac{\partial r_i}{\partial x_j}+\sum_{j=1}^n r_j\frac{\partial r_i}{\partial x_j}
\\
\\+\int_{{\mathbb R}^n}\left( \frac{\partial}{\partial x_i}K_n(x-y)\right) \sum_{j,k=1}^n\left( v^r_{k,j}v^r_{j,k}\right) (t,y)dy,\\
\\
-2\int_{{\mathbb R}^n}\left( \frac{\partial}{\partial x_i}K_n(x-y)\right) \sum_{j,k=1}^n\left( v^r_{k,j}r_{j,k}\right) (t,y)dy\\
\\
-\int_{{\mathbb R}^n}\left( \frac{\partial}{\partial x_i}K_n(x-y)\right) \sum_{j,k=1}^n\left( r_{k,j}r_{j,k}\right) (t,y)dy\\
\\
\mathbf{v}^r(0,.)=\mathbf{h}.
\end{array}\right.
\end{equation}
If we can solve this equation for $v^r_i\in C^{1,2}\left(\left[0,\infty\right)\times {\mathbb R}^n\right) , 1\leq i\leq n$ for an appropriate control function space, where we construct $r\in R$ time-step by time-step, then it is clear that  $v_i\in C^{1,2}\left(\left[0,\infty\right)\times {\mathbb R}^n\right)$ for $1\leq i\leq n$ is a global classical solution of the incompressible Navier Stokes equation. The construction is done time-step by time step on domains $\left[l-1,l\right]\times {\mathbb R}^n,~l\geq 0$ where the restriction of the contro function to $\left[l-1,l\right]\times {\mathbb R}^n$ is denoted by $r^l_i$. Here the local equation then is defined in terms of transformed time coordinates $t=\rho_l\tau$ , where $\rho_l\sim \frac{1}{l}$ and $\tau\in [l-1,l]$ at each time step. The local functions $v^{r,\rho,l}_i$ with $v^{r,\rho,l}_i(\tau,x)=v^{r,l}_i(t,x)$ are defined inductively on $\left[l-1,l\right]\times {\mathbb R}^n$ along with the control function $r^l$ via the Cauchy problem for
\begin{equation}
\mathbf{v}^{r,\rho,l}:=\mathbf{v}^{\rho,l}+\mathbf{r},
\end{equation}
where $\mathbf{v}^{\rho,l}=\left(v^{\rho,l}_1,\cdots ,v^{\rho,l}_n \right)^T$ is the time transformed solution of the incompressible Navier Stokes equation (in Leray projection form) restricted to the domain
$\left[l-1,l\right]\times {\mathbb R}^n$ where $v^{\rho,l}_i(\tau,x)=v^{l}_i(t,x)$ for $\tau\in [l-1,l]$ and $v^{l}_i,~1\leq i\leq n$ denotes the restriction solution of the incompressible Navier Stokes equation (in Leray projection form) to the domain
$\left[\sum_{m=1}^{l-1}\rho_m,\sum_{m=1}^l\rho_m\right]\times {\mathbb R}^n$. Here
where $\mathbf{v}=\left(v_1,\cdots v_n\right)^T$ and  $\mathbf{r}=\left(r_1,\cdots r_n\right)^T$. Then the function $\mathbf{v}^{r,\rho,l}=\left(v^{\rho,l}_1+r^l_1,\cdots v^{\rho,l}_n+r^l_n\right)^T$
satisfies the equation   
\begin{equation}\label{Navleraycontrolledl}
\left\lbrace \begin{array}{ll}
\frac{\partial v^{r,\rho,l}_i}{\partial \tau}-\rho_l\nu\sum_{j=1}^n \frac{\partial^2 v^{r,\rho,l}_i}{\partial x_j^2} 
+\rho_l\sum_{j=1}^n v^{r,\rho,l}_j\frac{\partial v^{r,\rho,l}_i}{\partial x_j}=\\
\\
\rho_l\frac{\partial r^l_i}{\partial t}-\rho_l\nu\sum_{j=1}^n \frac{\partial^2 r^l_i}{\partial x_j^2} 
+\rho_l\sum_{j=1}^n r^l_j\frac{\partial v^{r,\rho,l}_i}{\partial x_j}\\
\\
+\rho_l\sum_{j=1}^n v^{r,\rho,l}_j\frac{\partial r^l_i}{\partial x_j}+\rho_l\sum_{j=1}^n r^l_j\frac{\partial r^l_i}{\partial x_j}
\\
\\+\rho_l\int_{{\mathbb R}^n}\left( \frac{\partial}{\partial x_i}K_n(x-y)\right) \sum_{j,k=1}^n\left( \frac{\partial v^{r,\rho,l}_k}{\partial x_j}\frac{\partial v^{r,\rho,l}_j}{\partial x_k}\right) (\tau,y)dy,\\
\\
-2\rho_l\int_{{\mathbb R}^n}\left( \frac{\partial}{\partial x_i}K_n(x-y)\right) \sum_{j,k=1}^n\left( \frac{\partial v^{r,\rho,l}_k}{\partial x_j}\frac{\partial r^l_j}{\partial x_k}\right) (t,y)dy\\
\\
-\rho_l\int_{{\mathbb R}^n}\left( \frac{\partial}{\partial x_i}K_n(x-y)\right) \sum_{j,k=1}^n\left( \frac{\partial r^l_k}{\partial x_j}\frac{\partial r^l_j}{\partial x_k}\right) (\tau,y)dy\\
\\
\mathbf{v}^{r,\rho,l}(l-1,.)=\mathbf{v}^{r,\rho,l-1}(l-1,.).
\end{array}\right.
\end{equation}
In \cite{KNS} the idea for a construction of a control function $r^l_i$ was to solve at each time step $l\geq 1$ and for all $1\leq i\leq n$ a Cauchy problem for $r^l_i$ such that the rigth side of the first equation in (\ref{Navleraycontrolledl}) equals a source function $\phi^l_i$ where such a source function is chosen from the data obtained at the previous time step (resp. from the initial data at the first time step). For example this consumption function  may be chosen proportional to $-v^{r,\rho,l-1}_i(l-1)$ for all $1\leq i\leq n$. With the right choice of the time step size $\rho_l>0$ this ensures that the value function of the global controlled scheme is globally bounded. In this paper we choose another method which simplifies the construction of the control function, where we use the fact that it is sufficent to get a linear global bound for the Leray projection term in order to kmake the scheme global while choosing a time step size $\rho_l\sim\frac{1}{l}$.
The idea in this paper is to construct $r^l_i,1\leq i\leq n$ inductively with respect to the time step number $l$ via a local iterative solution scheme for (\ref{Navleraycontrolledl}) such that the global functions $r_i,~1\leq i\leq n$ are in the function space $R$ (defined above). This construction depends on the local solution scheme. This local solution scheme for the controlled function $v^{r,\rho,l}_i,~1\leq i\leq n$ approximates at each time step $l\geq 1$ this very function by a local functional series
\begin{equation}
v^{r,\rho,l}_i=v^{r,\rho,l,0}_i+\sum_{k=1}^{\infty}\delta v^{r,\rho,l,k}_i
\end{equation}
where the functional increments $\delta v^{r,\rho,l,k}_i=v^{r,\rho,l,k}_i-v^{r,\rho,l,k-1}_i$  satisfy some contraction property in a function space $L^{\infty}\cap C^1\times H^2\cap C^2_0$, where $C^2_0\equiv C^2_0\left({\mathbb R}^n\right)$ is the functions space of twice differentiable functions vanishing at infinity. Note that this latter space is closed with respect to the uniform norm. The idea of the control function is to define at each time step $l\geq 1$ the increment of the control function via the negative increment of the first linear approximation of the local functional series, i.e., to define
\begin{equation}
r^l_i(.,.)=r^{l-1}_i-\left(v^{*,\rho,1,l}_i(.,.)-v^{r,\rho,l-1}_i(l-1,.)\right), 
\end{equation}
where $v^{*,\rho,1,l}_i(.,.)$ satisfies
\begin{equation}\label{Navlerayuncontrolledl}
\left\lbrace \begin{array}{ll}
\frac{\partial v^{*,\rho,1,l}_i}{\partial \tau}-\rho_l\nu\sum_{j=1}^n \frac{\partial^2 v^{r,\rho,1,l}_i}{\partial x_j^2} 
+\rho_l\sum_{j=1}^n v^{r,\rho,l-1}_j\frac{\partial v^{*,\rho,1,l}_i}{\partial x_j}=\\
\\+\rho_l\int_{{\mathbb R}^n}\left( \frac{\partial}{\partial x_i}K_n(x-y)\right) \sum_{j,k=1}^n\left( \frac{\partial v^{*,\rho,l-1}_k}{\partial x_j}\frac{\partial v^{\rho,l-1}_j}{\partial x_k}\right) (l-1,y)dy,\\
\\
\mathbf{v}^{*,\rho,1,l}(l-1,.)=\mathbf{v}^{r,\rho,l-1}(l-1,.).
\end{array}\right.
\end{equation}

It turns out then that the higher order corrections $\sum_{k=1}^{\infty}\delta v^{r,\rho,l,k}_i$ contribute just to a linear growth of the Leray projection term for the solution function of the controlled system. This is sufficient in order to obtain a global scheme with a regular solution function since we may choose a time step size of order $\rho_l\sim \frac{1}{l}$.
This is roughly the program of this paper. In the next section we reconsider the global scheme for the multivariate Burgers equation. Then in the last section we extend our considerations to the incompressible Navier-Stokes equation.

\section{The global scheme for the multivariate Burgers equation}
Let us first reconsider the constructive approach to global existence for the Cauchy problem of the multidimensional Burgers equation in (\ref{qparasyst1}). 
The following considerations apply also to a class of initial-value boundary problems on different domains $[0,\infty)\times \Omega$ with $\Omega\subseteq {\mathbb R}^n$.
We assume smooth initial data $\mathbf{h}=(h_1,\cdots,h_n)^T$ where for all $1\leq i\leq n$ the functions $h_i$ have polynomial decay at infinity- this implies that the assumption $h_i\in H^s$ holds for all $1\leq i\leq n$ and $s\in {\mathbb R}$. We review some ideas in \cite{KB1} from a slightly different point of view. Only ideas of proofs are given. We provide more details when we consider the extension to the incompressible Navier-Stokes equation.
The idea of the constructive approach to global existence and regularity for this Cauchy problem and related Cauchy problems is to set up an iteration scheme with respect to time $\tau=\rho_l t$ at each time step $l\geq 1$ and with a time step size
\begin{equation}
\rho_l\sim \frac{1}{l},
\end{equation}
and such that at each time step an equation
\begin{equation}\label{qparasyst2}
\left\lbrace \begin{array}{ll}
\frac{\partial u^{\rho,l}_i}{\partial \tau}=\rho_l\nu\sum_{j=1}^n \frac{\partial^2 u^{\rho,l}_i}{\partial x_j^2} 
-\rho_l\sum_{j=1}^n u^{\rho,l}_j\frac{\partial u^{\rho,l}_i}{\partial x_j},\\
\\
\mathbf{u}^{\rho,l}(l-1,.)=\mathbf{u}^{\rho,l-1}(l-1,.)
\end{array}\right.
\end{equation} 
is solved on the domain $[l-1,l]\times {\mathbb R}^n$ by a time-local iterative scheme. he scheme becomes global if we can establish a linear bound for the value function. The solution is constructed via a functional series
\begin{equation}\label{funci2}
u^{\rho ,l }_i=u^{\rho,1,l}_i+\sum_k \delta u^{\rho, k+1,l}_i, 1\leq i\leq n,
\end{equation} 
 where $u^{\rho,1,l}_i$ solves 
 \begin{equation}\label{scalparasystlin10}
\left\lbrace \begin{array}{ll}
\frac{\partial u^{\rho,1,l}_i}{\partial \tau}=\rho_l\left( \nu\sum_{j=1}^n \frac{\partial^2 u^{\rho,1,l}_i}{\partial x_j^2} 
-\sum_{j=1}^n u^{\rho,l-1}_j(l-1,.)\frac{\partial u^{\rho,1,l}_i}{\partial x_j}\right) ,~~1\leq i\leq n,\\
\\
{\bf u}^{\rho,1,l}(l-1,.)={\bf u}^{l-1}(l-1,.),
\end{array}\right.
\end{equation} 
and $\delta u^{\rho,k,l}_i,~1\leq i\leq n$ solves
\begin{equation}\label{deltaurhok0}
\left\lbrace \begin{array}{ll}
\frac{\partial \delta u^{\rho,k+1,l}_i}{\partial \tau}=\rho_l\left( \sum_{j=1}^n \frac{\partial^2 \delta u^{\rho,k+1,l}_i}{\partial x_j^2} 
-\sum_{j=1}^n u^{\rho,k,l}_j\frac{\partial \delta u^{\rho,k+1,l}_i}{\partial x_j}\right) \\
\\
\hspace{2cm}-\rho_1\sum_j\left(\delta u^{\rho,k,l}_j\frac{\partial u^{\rho,k,l}}{\partial x_j} \right),\\ 
\\
\mathbf{\delta u}^{\rho,k+1,l}(l-1,.)= 0,
\end{array}\right.
\end{equation}
and where $\delta u^{\rho,k+1,l}_{i}=u^{\rho,k+1,l}_{i}-u^{\rho,k,l}_{i}$ for $k\geq 1$, and 
$\delta u^{\rho,1,1}_j= u^{\rho,1,1}_j-h_j$ at the first time step.
 Although the representations of the classical solutions of the members of the functional series ($\delta u^{\rho,k,l}_i$) in terms of fundamental solutions are not convolutions we may estimate them by convolutions using Gaussian a priori estimates of fundamental solutions. In order to do this we may use the Levy expansion of fundamental solutions as we did in \cite{KB1} or we may use properties of the adjoint.  For the estimation of convolutions  we use the generalized Young inequality (\ref{Y1}) and (\ref{Y2}) above.
This leads to
\begin{lem}
 Let $u^{\rho,l-1}_i(l-1,.)\in H^s$ for all $1\leq i\leq n$ and for any given $s\in {\mathbb R}$. Then 
 \begin{equation}
 u^{\rho, 1,l}_i(\tau,.)\in H^s~\mbox{and}~\delta u^{\rho, k,l}_i(\tau,.)\in H^s
 \end{equation}
uniformly in $\tau$ and for all $1\leq i\leq n$ and all $s\in {\mathbb R}$. Furthermore,
\begin{equation}
u^{\rho, l}_i(\tau,.):=u^{\rho, 1,l}_i(\tau,.)+\sum_{k=1}^{\infty}\delta u^{\rho,k,l}_i\in H^s.
\end{equation}
\end{lem}
\begin{proof}We consider the case $n=3$ (the generalisation to $n>3$ is similar).
This is proved for $H^2$ first with classical representations of the fundamental solution. The latter exists for $n=3$ since first order coefficients for the equation for $u^{\rho,1,l}_i$ are of form $u^{\rho,l-1}_j\in H^2$ and hence H\"{o}lder continuous. Inductively with the substep number $k\geq 1$ the same is true for the first order coefficients of the equation for $\delta u^{\rho,k+1,l}_i$. Then we may represent the solutions  for  $u^{\rho,1,l}_i$ and for $\delta u^{\rho,k+1,l}_i$ in terms of the fundamental solutions. Using the adjoint and partial integration we may represent the approximative solution function and its derivatives up to second order by a representation which contains only first order derivatives of the respective fundamental solutions (cf. \cite{KB1} of this paper or the next section of this paper). Then derivatives of first order of the fundamental solution have time-integrable Gaussian a priori estimates. We can write down estimates which are convolutions, and then we may use the generalized Young inequality in order to establish $L^2$ estimates for these convolutions (using $L^1$ estmates of the Gaussian). Note that locally we may use the local bound
\begin{equation}
\frac{C}{(t-s)^{\mu}(x-y)^{n+1-2\mu}}
\end{equation}
for $0.5< \mu< 1$ for the first derivative of the fundamental solution. 
 From the classical representations of the functions $\delta u^{\rho,k,l}$ we get a contraction property with respect to the $H^2$-norm uniformly  $\tau$. This implies that the local limit function ($k\uparrow \infty$) is H\"{o}lder, and this implies that the limit function has a representations in terms of a fundamental solution. Higher order regularity may then be obtained by differentiation shifting derivative to approximative value functions and then using Gaussian estimates and the generalized Young inequality, and the adjoint again.
\end{proof}
Now in addition we may prove that $u^{\rho,l}_i$ are differentiable with respect to $\tau$. Furthermore we need a bound on the first order coefficients $u^{\rho,l-1}_i$ which is independent of the time step number $l$. Indeed we have
\begin{lem}
 \begin{equation}
 |u^{\rho,l-1}_i(l-1,.)|_{\alpha}\leq C \mbox{ for all } 1\leq i\leq n 
\end{equation}
for a constant $C>0$ which is independent of the time step number $l\geq 1$.
\end{lem}
\begin{proof}Again we consider the case $n=3$
 We observe that the growth of $u^{\rho,l}$ is linear with respect to the time step number $l$. At each time step a time step size of order $\rho_l\sim \frac{1}{l}$ then ensures that the first order coefficients $u^{\rho,l-1}_i(l-1,.)$ for the equations for $u^{\rho,1,l}_i$ have a uniform bound (independent of the time step number $l$).
\end{proof}
Invoking classical regularity results for linear parabolic equations and induction over $l$ yields
 \begin{thm}\label{Burgthm}
 The Cauchy problem (\ref{qparasyst1}) on $\left[0,\infty \right)\times {\mathbb R}^n$ with initial data
 $h_i\in H^s$ for all $1\leq i\leq n$ and all $s\in {\mathbb R}$ has a global regular solution
 \begin{equation}
 u_i\in C^{\infty}\left(\left[0,\infty\right) \times {\mathbb R}^n\right) 
 \end{equation}
 along with $u(t,.)\in H^s$ for all $t\in [0,\infty)$.  
 \end{thm}
This is not a new result but the proof is more elementary compared to an alternative method which establishes first the estimate
\begin{equation}\label{aprioriest}
\frac{\partial}{\partial t}\|u(t,.)\|_{H^s}\leq \|u(t,.)\|_{H^{s+1}}\sum_{i,j}\sum_{|\alpha|+|\beta|\leq s}\|D^{\alpha}u_iD^{\beta}u_j\|_{L^2}-2\|\nabla u\|^2_{H^s}.
\end{equation}
At first glance from the construction it seems that the estimate for the  solution function $\mathbf{u}$ increases linearly on a time scale $\sum_{k=1}^l\rho_k$ with $\rho_k\sim \frac{1}{k}$, and, hence, quadratically on an uniform time scale. However, at each time step, if we have constructed $u^{\rho,l}_i(\tau,.)\in H^s$ for $1\leq i\leq n$ uniformly in $\tau$ with $\tau\in [l-1,l]$, then the construction gives
\begin{equation}
|u^{\rho,l}_i(\tau,.)|_{H^k}\leq |u^{\rho,l-1}_i(\tau,.)|_{H^k}+C_k
\end{equation}
with a constant $C_k$ depending on $k$ according to the level of differentiability we want to obtain (accordingly the stepsize $\rho_l$ has to be chosen for each given $k$). However, the growth of size $C_k$ at each time step can be compensated by a choice $\rho_l\sim \frac{1}{l C_k}$. Furthermore, we can easily deduce the existence of classical solutions $u^{\rho,l}$ of (\ref{qparasyst2}) at each time step. First we choose $k>m+\frac{1}{2}n$ in the construction above and obtain spatial differentiability of order $m$ uniformly with respect to $\tau$. Then the product estimate
\begin{equation}\label{prule}
 |fg|_{H^s}\leq C_s|f|_{H_s}|g|_{H_s}~\mbox{ for }~s>\frac{1}{2}n
\end{equation}
 gives
\begin{equation}
\frac{\partial u^{\rho,l}_i}{\partial \tau}=\rho_l\nu\sum_{j=1}^n \frac{\partial^2 u^{\rho,l}_i}{\partial x_j^2} 
-\rho_l\sum_{j=1}^n u^{\rho,l}_j\frac{\partial u^{\rho,l}_i}{\partial x_j}\in H^{k-2}.
\end{equation}
 Hence for choice $k>2+\frac{1}{2}n$ we have differentiability of $u^{\rho,l}_i$ with respect to time. The existence of a fundamental solution of the equation
 \begin{equation}
 \frac{\partial p^{\rho,l}}{\partial \tau}=\rho_l\nu\sum_{j=1}^n \frac{\partial^2 p^{\rho,l}}{\partial x_j^2} 
-\rho_l\sum_{j=1}^n u^{\rho,l}_j\frac{\partial p^{\rho,l}}{\partial x_j}
 \end{equation}
 is ensured in terms of the known functions $u^{\rho,l}_j$ which are indeed now known to be H\"{o}lder continuous with respect to the time argument $\tau$ and the spatial argument $x$. Hence we have classical representations of the solution and an ordinary maximum principle for linear parabolic equations tells us that the maximum of the solution function $u^{\rho,l}$ does not increase over the time  interval $[l-1,l]$. As the step size compensates the linear growth of the estimate immediately linked to the original construction this holds independently of the time step number $l$. We note
 \begin{cor}(Same assumptions as in theorem \ref{Burgthm} above).
 
 The solution $\mathbf{u}$ of (\ref{qparasyst1}) is bounded. 
 \end{cor}

 \section{Extension of the scheme to the incompressible Navier Stokes equation}
Next we look at the relation to the incompressible Navier-Stokes equation. We write it in the same scheme frame as above with time step number $l\geq 1$ and time step size $\rho_l\sim \frac{1}{l}$ at each time step number $l$. The velocity is denoted by $\mathbf{v}=(v_1,\cdots ,v_n)^T$ (recall that $T$ is for 'transposed'). We have an additional scalar function $p$, and in time-transformed coordinates $\tau=\rho_l t$ the equation system scheme becomes 
\begin{equation}\label{qparasystnav1}
\left\lbrace \begin{array}{ll}
\frac{\partial v^{\rho,l}_i}{\partial \tau}=\rho_l\nu\sum_{j=1}^n \frac{\partial^2 u^{\rho,l}_i}{\partial x_j^2} 
-\rho_l\sum_{j=1}^n v^{\rho,l}_j\frac{\partial v^{\rho,l}_i}{\partial x_j}-\rho_l\frac{\partial p^{\rho,l}}{\partial x_i},\\
\\
\div \mathbf{v}^{\rho,l}=0
\\
\\
\mathbf{v}^{\rho,l}(l-1,.)=\mathbf{v}^{\rho,l-1}(l-1,.),
\end{array}\right.
\end{equation}  
where for each $l\geq 1$ we consider this system on $[l-1,l]\times {\mathbb R}^n$. 
For $l=1$ the initial conditions are $\mathbf{v}^{\rho,1}(0,.)=\mathbf{v}^{\rho,l-1}(l-1,.)=\mathbf{h}(.)$ with the same $\mathbf{h}\in H^s\left({\mathbb R}^n\right) $ for all $s\in {\mathbb R}$ as in the preceding section.
The Leray projection form of these equations is obtained time-step by time-step by elimination of $p^{\rho,l}$ via an equation for $\div \mathbf{v}^{\rho,l}$ which simplifies by incompressibility (at each time step $l\geq 1$). First, Leray projection leads to the Poisson equation for $p^{\rho,l}(\tau,.)$ for $\tau\in [l-1,l]$, i.e.,
\begin{equation}
\Delta p^{\rho,l}=-\sum_{i,j=1}^n \frac{\partial v^{\rho,l}_i}{\partial x_j}\frac{\partial v^{\rho,l}_j}{\partial x_i},
\end{equation}
where we suppress the notation of the evaluation at $\tau$ ($\tau$ serves as a parameter). Hence
at each time step we have the local equation in Leray projection form of the incompressible Navier-Stokes equation. If
\begin{equation}
\sum_{i,j=1}^n \left( \frac{\partial v^{\rho,l}_i}{\partial x_j}\frac{\partial v^{\rho,l}_j}{\partial x_i}\right)(\tau,.)
\end{equation}
is in $C^1\left({\mathbb R}^n\right) \cap L^1\left({\mathbb R}^n\right)$ for all $\tau\in [l-1,l]$, then this is justified  by a well-known result. Indeed we have
\begin{prop}\label{stprop}
Assume that $n\geq 3$ and that $f\in C^1\left({\mathbb R}^n\right) \cap L^1\left({\mathbb R}^n\right)$ and let $K$ be the fundamental solution of the Laplacian, i.e., of the equation
\begin{equation}
\Delta K=\delta,
\end{equation}
where $\delta$ denotes the Dirac distribution. Then for $n\geq 3$
\begin{equation}
K(x)=-\frac{|x|^{2-n}}{(2-n)\omega_n}
\end{equation}
(with $\omega_n$ being the surface area of the unit sphere) defines a distributive solution and determines a classical solution $w\in C^2$ defined by
\begin{equation}
w(x):=\int_{{\mathbb R}^n}f(y)K(x-y)dy
\end{equation}
of the Poisson equation
\begin{equation}
\Delta w=f.
\end{equation}
Moreover, the gradient
\begin{equation}
\nabla w(x):=\int_{{\mathbb R}^n}f(y)\nabla K(x-y)dy
\end{equation}
is well-defined.
\end{prop}
Note that  in our scheme $L^2$-theory for the data is enough in order to use this proposition \ref{stprop} since
\begin{equation}\label{prodv}
{\Big |}\frac{\partial v^{\rho,l}_i}{\partial x_j}(\tau,.)\frac{\partial v^{\rho,l}_j}{\partial x_i}(\tau,.){\Big |}\leq \frac{1}{2}\left( \frac{\partial v^{\rho,l}_i}{\partial x_j}\right)^2(\tau,.)+\frac{1}{2}\left( \frac{\partial v^{\rho,l}_j}{\partial x_i}\right)^2(\tau,.)
\end{equation}
for any $\tau \in [l-1,l]$. This means that $v^{\rho,l}_i\in H^1$ implies that the left side of (\ref{prodv}) is in $L^1$. Outside a ball we may then use that the first derivatives of the kernel $K_{,i}$ are in $L^2$ in order to estimate pressure terms. The truncation of the kernel inside a ball convoluted with the data is easier to handle. Indeed since the data are in our scheme at time step $l\geq 1$, i.e.,  $v^{\rho,l}_j(l-1,.)$, are in $H^2\cap C^2$, and we have
  \begin{equation}\label{prodv2}
\sum_{i,j=1}^n{\Big |}\frac{\partial v^{\rho,l}_i}{\partial x_j}(l-1,.)\frac{\partial v^{\rho,l}_j}{\partial x_i}(l-1,.){\Big |}
\leq nC\sum_{j=1}^n{\Big |}\frac{\partial v^{\rho,l}_j}{\partial x_i}(l-1,.){\Big |}.
\end{equation}
We may then use that the right side of (\ref{prodv2}) is in $L^2$ and together with the fact that the truncated kernel $1_{B_r(0)}K_{,i}$ is in $L^1$ the generalized Young inequality leads to the conclusion that the convolution with truncated kernel is also in $L^2$. Similar for first derivatives.
Before we start to go into details, let us outline the program of this paper. It makes sense to consider the Leray projection term first and its properties. This term makes the difference to the multivariate Burgers equation, and approximations appear in our local iterative scheme of course. We first consider $H^k$-estimates for the Leray projection term. Then we consider time-local contraction results for the uncontrolled scheme with approximations $v^{\rho,k,l}_i$ and with respect to a $L^{\infty}\times H^2$-norm. In addition we consider time-local contractions with respect to a $L^{\infty}\times H^{2,\infty}$ norm. It makes sense to consider both local contraction results together since this simplifies the global estimates for the controlled scheme. This time-local contraction results imply local existence, and this estimates can be repeated for an extended controlled scheme. We shall do this and observe that for a certain choice of the control functions $r^l_i$ at each time step $l$ we get a global linear bound of the Leray projection term. This is the technical abstract. We next go into the details.    
  Next we extend the standard result proposition \ref{stprop} in order to meet our purposes. It is useful to have the gradient of the pressure in $L^2$ (resp. in $H^1$) since estimates in terms of convolutions of the modulus of the gradient with a Gaussian a priori estimate  help us to conclude that the convolution bound is itself in $L^2$ (resp. in $H^1$). Note that the relationships in (\ref{Y1}) and (\ref{Y2}) make $L^2$-estimates for the gradient of the pressure useful in order to get $L^2$-estimates for the approximations $v^{\rho,k,l}_i$ of our scheme and their (weak) derivatives up to second order. Furthermore, if we want $L^{\infty}$-estimates for approximations $v^{\rho,k,l}_i$ and their (weak) derivatives up to second order, then $L^2$- estimates for the gradient of the pressure are also useful since the Gaussian is in $L^2$.  The next lemma is rather simple bt we consider this in more detail because we use these simple observations in a more complex situation below.
\begin{lem}\label{poisson}
Let $n=3$.
For $k\geq 1$ assume that $f_i,g_j\in H^{k+1}\cap C^{k+1}$ for $1\leq i\leq n$, and define for all $x\in {\mathbb R}^n$
\begin{equation}\label{nablap}
\frac{\partial}{\partial x_i}p^{f,g}(x)=\int_{{\mathbb R}^n}\sum_{i,j=1}^n \left( \frac{\partial f_i}{\partial x_j}\frac{\partial g_j}{\partial x_i}\right) (y) \frac{\partial}{\partial x_i}K(x-y)dy.
\end{equation}
Then we have $\frac{\partial}{\partial x_i}p^{f,g}\in H^{k+1}\cap C^{k+1}$ and  $p^{f,g}\in H^{k+2}\cap C^{k+2}$ (defined accordingly). We have the estimate
\begin{equation}
{\Big |}\frac{\partial}{\partial x_i}p^{f,g}{\Big |}_{H^{k+1}}\leq  C\max_{i,j\in \left\lbrace 1,\cdots ,n\right\rbrace} \left( |f_i|_{H^{k}}+|g_j|_{H^{k}}\right)
\end{equation}
for some constant $C>0$ which depends on the bounded first order derivatives of either $g_k$ or $f_j$. 
\end{lem}

\begin{rem}
Note that for a constant $C>0$ independent of $f_j$ and $g_k$ we have the estimate
\begin{equation}
{\Big |}\frac{\partial}{\partial x_i}p^{f,g}{\Big |}_{H^{k+1}}\leq  C\max_{i,j\in \left\lbrace 1,\cdots ,n\right\rbrace}\left( |f_i|_{H^{k}}+|g_j|_{H^{k}}+
|f_i|_{H^{k}}|g_j|_{H^{k}}\right)
\end{equation}
\end{rem}

\begin{proof}
We consider the case $f,g\in H^2\cap C^2$ and $k=1$ which is essential. For $k>0$ an analougous argument is obtained by multivariate differentiation of the equivalent expression 
\begin{equation}\label{nablapeq}
\frac{\partial}{\partial x_i}p^{f,g}(x)=\int_{{\mathbb R}^n}\sum_{i,j=1}^n \left( \frac{\partial f_i}{\partial x_j}\frac{\partial g_j}{\partial x_i}\right) (x-y) \frac{\partial}{\partial x_i}K(y)dy.
\end{equation} 
 Note that the second derivatives of $f$ and $g$ are in $L^2\cap C$, and therefore in the closed space of continuous functions vanishing at infinity. Hence they are bounded.
The basic idea is is the following. The representation of the gradient of the pressure is a convolution of a sum of products
\begin{equation}\label{fgproduct}
 \sum_{i,j=1}^n  \frac{\partial f_i}{\partial x_j}\frac{\partial g_j}{\partial x_i}
\end{equation}
with the gradient of the kernel $K$.
Since for all $x\in {\mathbb R}^n$ we have
\begin{equation}\label{prodfg}
\begin{array}{ll}
\sum_{i,j=1}^n {\Big |} \frac{\partial f_i}{\partial x_j}\frac{\partial g_j}{\partial x_i}{\Big |}(x)
\leq\sum_{j,k=1}^n \frac{1}{2}\left( \frac{\partial f_j}{\partial x_k}\right)^2(x)+
\frac{1}{2}\left( \frac{\partial g_k}{\partial x_j}\right)^2 (x),
\end{array}
\end{equation}
it follows that the (sum of) product(s) function in (\ref{fgproduct}) is in $H^{1,1}$ (the Sobolev space of functions in $L^1$ with weal derivatives in $L^1$). Accordingly the first order derivatives are in $L^1$ (since $f_j$ and $g_k$ are in $H^2$).  In order to have the gradient of the pressure in $H^1$ the next idea is to split up the kernel $K$ (or a first order derivative of it) into two summands writing the gradient of the pressure functional as
\begin{equation}\label{nablap1}
\begin{array}{ll}
\frac{\partial}{\partial x_i}p^{f,g}(x)=\int_{{\mathbb R}^n}\sum_{j,k=1}^n \left( \frac{\partial f_j}{\partial x_k}\frac{\partial g_k}{\partial x_j}\right) (x-y) K_{,i}(y)dy\\
\\
=\int_{{\mathbb R}^n}\sum_{j,k=1}^n \left( \frac{\partial f_j}{\partial x_k}\frac{\partial g_k}{\partial x_j}\right) (x-y) \phi_{B_r}(y)K_{,i}(y)dy\\
\\
+\int_{{\mathbb R}^n}\sum_{j,k=1}^n \left( \frac{\partial f_j}{\partial x_k}\frac{\partial g_k}{\partial x_j}\right) (x-y) (1-\phi_{B_r})(y)K_{,i}(y)dy,
\end{array}
\end{equation}
along with a smooth function supported on a ball of radius $r>0$ around $B_r(0)$ and zero elsewhere. Recall the Einstein notation where $K_{,i}$ denotes the partial $i$th derivative of the kernel function $K$. 
Note that for $n=3$ we have
\begin{equation}\label{delocK}
y\rightarrow (1-\phi_{B_r})(y)K_{,i}(y)\in L^2
\end{equation}
and 
\begin{equation}\label{locK}
y\rightarrow \phi_{B_r}(y)K_{,i}(y)\in L^1.
\end{equation}
Here for (\ref{delocK}) observe that the function $(1-\phi_{B_r})K_{,i}$ is square integrable, since for $n=3$ we have
\begin{equation}
\int_{{\mathbb R}^n\setminus B_r(0)}(1-\phi_{B_r}(x))\frac{x_i^2}{|x|^{2n}}dx\sim \int_{s\geq r}\frac{s^2}{s^6}s^2ds<\infty.
\end{equation}
For (\ref{locK}) observe that for $n=3$
\begin{equation}
K_{,i}(y)\sim \frac{y_i}{|y|^n}\sim \frac{r}{r^3}
\end{equation}
is locally integrable (but not square integrable).
The function $\phi_{B_r}(x))$ may be specified in the form
\begin{equation}
\phi_{\epsilon}(y)=\left\lbrace \begin{array}{ll}
         1 \mbox{ if } |y|\leq \epsilon ,\\
	 \\
	 \exp\left(-\frac{1}{(2\epsilon^2-|y|^2)} \right) \mbox{ if } \epsilon^2\leq |y|^2\leq 2\epsilon^2,\\
	 \\
	 0 \mbox{ else}.
        \end{array}\right.
\end{equation}
For the second term on the right side of (\ref{nablap1}) we have the inequality
\begin{equation} 
\begin{array}{ll}
\int_{{\mathbb R}^n}\sum_{j,k=1}^n \left( \frac{\partial f_j}{\partial x_k}\frac{\partial g_k}{\partial x_j}\right) (x-y) (1-\phi_{B_r})(y)K_{,i}(y)dy\\
\\
\leq \int_{{\mathbb R}^n}\sum_{j,k=1}^n \frac{1}{2}\left( {\Big |} \frac{\partial f_j}{\partial x_k}{\Big |}^2+{\Big |}\frac{\partial g_k}{\partial x_j}{\Big |}^2\right)  (x-y) (1-\phi_{B_r})(y)K_{,i}(y)dy\\
\\
\leq C\sum_{j=1}^n\left( |f_j|_{H^1}+ |g_j|_{H^1}\right) 
\end{array}
\end{equation}
where we use the Young inequality.
Next we consider the first term on the right side of (\ref{nablap1}). Since
$f_j\in C^1\cap H^1$ the first order derivatives of $f_j, 1\leq j\leq n$ are bounded by a constant $C>0$, hence we have
\begin{equation}\label{firstterm}
\begin{array}{ll}
 \int_{{\mathbb R}^n}\sum_{j,k=1}^n \left( \frac{\partial f_j}{\partial x_k}\frac{\partial g_k}{\partial x_j}\right) (x-y) \phi_{B_r}(y)K_{,i}(y)dy\\
\\
\leq \int_{{\mathbb R}^n}\sum_{j,k=1}^n C{\Big |}\frac{\partial g_k}{\partial x_j}{\Big |} (x-y) \phi_{B_r}(y)K_{,i}(y)dy
\end{array}
\end{equation}
Since $g_k\in H^1$ we have $g_{k,j}\in L^2$ and we may use the Young inequality along with $\phi_{B_r}(.)K_{,i}(.)\in L^1$ in order to obtain for a generic $C>0$ that the right side of (\ref{firstterm}) has the upper bound
\begin{equation}\label{ineqf}
C\sum_{k=1}^n|g_k|_{H^1}
\end{equation}
for some generic constant $C>0$. 
Symmetrically one has the upper bound
\begin{equation}\label{ineqg}
C\sum_{j=1}^n|f_j|_{H^1},
\end{equation}
of course. Note that the constant $C$ in (\ref{ineqf}) is proportional to $\max_{j\in \left\lbrace 1,\cdots ,n\right\rbrace }|f_j|_{H^1}$ while the constant in (\ref{ineqg}) is essentially proportional to  $\max_{j\in \left\lbrace 1,\cdots ,n\right\rbrace }|g_j|_{H^1}$. Hence we get the upper bound
\begin{equation}
C_0\left( \max_{j\in \left\lbrace 1,\cdots ,n\right\rbrace }|g_j|_{H^1}\max_{k\in \left\lbrace 1,\cdots ,n\right\rbrace }|f_k|_{H^1}\right) 
\end{equation}
for some constant $C_0$ which depends only on the dimansion and the Laplacian kernel.

\end{proof}

We need also standard Gaussian estimate for the fundamental solution and its first derivatives. We have
\begin{lem}\label{lemgauss}
Let $D:=[T_0,T_1]\times \overline{\Omega}\subseteq {\mathbb R}^n$ be a domain along with $T_1>T_0>0$, and let 
\begin{equation}\label{l}
Lp\equiv \frac{\partial p}{\partial t}-\sum_{i,j=1}^na_{ij}\frac{\partial^2 p}{\partial x_i\partial x_j}+\sum_{i=1}^nb_i\frac{\partial p}{\partial x_i}=0
\end{equation}
be an equation which satisfies
\begin{itemize}
 \item[i)] $L$ is uniformly parabolic on the whole of $[T_0,T_1]\times \Omega$,
 \item[ii)] the coefficient functions $a_{ij}$ are uniformly H\"{o}lder continuous with H\"{o}lder constant $\alpha/2\in (0,0.5)$ with respect to time and H\"{o}lder constant $\alpha\in (0,1)$ with respect to the spatial variables, i.e., $a_{ij}\in C^{\alpha/2,\alpha}\left(D\right)$,
 \item[iii)] the coefficient functions are H\"{o}lder continuous with H\"{o}lder constant $\alpha\in (0,1)$ and uniformly with respect to time $t$.
\end{itemize}
Then a fundamental solution $p$ of equation (\ref{l}) exists and satisfies the Gaussian a priori 
estimates
\begin{equation}\label{apriori}
|p(t,x;s,y)|\leq \frac{C}{(t-s)^{n/2}}\exp\left(-\frac{\lambda(x-y)^2}{4(t-s)} \right),
\end{equation}
and
\begin{equation}\label{apriorider}
{\Bigg|}\frac{\partial}{\partial x_i}p(t,x;s,y){\Bigg|}\leq \frac{C}{(t-s)^{(n+1)/2}}\exp\left(-\frac{\lambda(x-y)^2}{4(t-s)} \right),
\end{equation}
for some constants $C>0$ and $\lambda>0$ ($\lambda$ less or equal to the lower ellipticity constant in general). Note that for $t>s$ these a priori bounds (as functions of $x-y$) with  have a uniform bound in $L^1$. 
\end{lem}

We also use some related standard results on the adjoint. We have
\begin{lem}\label{lemgaussad}
Let the assumptions of the preceding lemma be satisfied on a domain $D=[T_0,T_1]\times {\mathbb R}^n$, and that in addition we have 
\begin{itemize}
 \item[i)] The coefficient functions $a_{ij}$ and their first and second derivatives are bounded continuous functions on $D$,
 \item[ii)] the coefficient functions $b_i$ and their first and second derivatives are bounded continuous functions on $D$.
\end{itemize}
  Then the fundamental solution $p^*$ of the adjoint equation $L^*p*=0$ exists. Furthermore, for $t>s$
 \begin{equation}
 p^*(s,y;t,x)=p(t,x;s,y)
 \end{equation}
and analogous relations for the partial derivatives hold. Moreover,
and satisfies the Gaussian a priori 
estimates
\begin{equation}
|p^*(t,x;s,y)|\leq \frac{C}{(t-s)^{n/2}}\exp\left(-\frac{\lambda(x-y)^2}{4(t-s)} \right),
\end{equation}
and
\begin{equation}
{\Bigg|}\frac{\partial}{\partial x_i}p^*(t,x;s,y){\Bigg|}\leq \frac{C}{(t-s)^{(n+1)/2}}\exp\left(-\frac{\lambda(x-y)^2}{4(t-s)} \right),
\end{equation}
for some constants $C>0$ and $\lambda>0$ ($\lambda$ less or equal to the lower ellipticity constant in general).
\end{lem}
Concerning these lemmas we note that in our scheme only the first order coefficients are variable. The useful relationship of the fundamental solution and its adjoint may also be verified directly using the Levy expansion as we did in \cite{KB1}. We also discussed the matter in \cite{KNS}. Our lemma \ref{poisson} shows that we may base a global scheme for the incompressible Navier-Stokes equation on approximations in Leray projection form which are $H^2$ with respect to the spatial variables. The scheme is in $H^{2}$ for fixed time with respect to the spatial variables (which is easier to see), and this helps us in order to estimate the gradient of the pressure functions of type $p^{\rho,f,g}$ in $H^1$ as in the lemma \ref{poisson} above. However we need to show that in addition we have some local contraction for the functional increments $\delta v^{\rho,l,k}_i$ as a sequence with respect to the index $k$.

Next, for each $l\geq 1$ we consider the Leray projection form of the incompressible Navier-Stokes equation at each time step $l\geq 1$, i.e.,
\begin{equation}\label{qparasystnav2l}
\left\lbrace \begin{array}{ll}
\frac{\partial v^{\rho,l}_i}{\partial \tau}=\rho_l\nu\sum_{j=1}^n \frac{\partial^2 v^{\rho,l}_i}{\partial x_j^2} 
-\rho_l\sum_{j=1}^n v^{\rho,l}_j\frac{\partial v^{\rho,l}_i}{\partial x_j}\\
\\
+\rho_l\int_{{\mathbb R}^n}\sum_{i,j=1}^n \left( \frac{\partial v^{\rho,l}_i}{\partial x_j}\frac{\partial v^{\rho,l}_j}{\partial x_i}\right) (\tau,y)\frac{\partial}{\partial x_i}K_n(x-y)dy,\\
\\
\mathbf{v}^{\rho,l}(l-1,.)=\mathbf{v}^{\rho,l-1}(l-1,.).
\end{array}\right.
\end{equation}  
Next we may define a scheme similar to the scheme of the multivariate Burgers equation and with time step size $\rho_l\sim\frac{1}{l}$. Having computed $v^{\rho,l-1}_i(l-1,.)\in H^2\cap C^2$ we set up a local iteration scheme. In general, for $n\geq 3$ a condition of form $v^{\rho,l-1}_i(l-1,.)\in H^m\cap C^2$ for an integer $m$ with $m>\frac{1}{2}n$ is an appropriate choice. But this is quite similar. Hence, think of $n=3$ and of data and functional series approximations with $v^{\rho,k,l}_i(\tau.)\in H^2$ first. 
Again, the solution in Leray-projection form is constructed via a functional series
\begin{equation}\label{funciv}
v^{\rho ,l }_i=v^{\rho,1,l}_i+\sum_{k=1}^{\infty} \delta v^{\rho, k+1,l}_i, 1\leq i\leq n,
\end{equation} 
 where $v^{\rho,1,l}_i$ solves 
 \begin{equation}\label{scalparasystlin10v}
\left\lbrace \begin{array}{ll}
\frac{\partial v^{\rho,1,l}_i}{\partial \tau}-\rho_l\left( \nu\sum_{j=1}^n \frac{\partial^2 v^{\rho,1,l}_i}{\partial x_j^2} 
-\sum_{j=1}^n v^{\rho,l-1}_j(l-1,.)\frac{\partial v^{\rho,1,l}_i}{\partial x_j}\right)\\
\\
=\rho_l\int_{{\mathbb R}^n}\sum_{j,m=1}^n \left( \frac{\partial v^{\rho,l-1}_j}{\partial x_m}\frac{\partial v^{\rho,l-1}_m}{\partial x_j}\right) (l-1,y)\frac{\partial}{\partial x_i}K_n(x-y)dy,\\
\\
{\bf v}^{\rho,1,l}(l-1,.)={\bf v}^{l-1}(l-1,.),
\end{array}\right.
\end{equation} 
and $\delta v^{\rho,k,l}_i=v^{\rho,k,l}_i-v^{\rho,k-1,l}_i,~1\leq i\leq n$ solves
\begin{equation}\label{deltaurhok0}
\left\lbrace \begin{array}{ll}
\frac{\partial \delta v^{\rho,k+1,l}_i}{\partial \tau}-\rho_l\left( \sum_{j=1}^n \frac{\partial^2 \delta v^{\rho,k+1,l}_i}{\partial x_j^2} 
-\sum_{j=1}^n v^{\rho,k,l}_j\frac{\partial \delta v^{\rho,k+1,l}_i}{\partial x_j}\right)=\\
\\
-\rho_1\sum_j\left(\delta v^{\rho,k,l}_j\frac{\partial v^{\rho,k-1,l}}{\partial x_j} \right)\\ 
\\
-\rho_l\int_{{\mathbb R}^n}K_{n,i}(x-y){\Big (} \left( \sum_{j,m=1}^n\left( v^{\rho,k,l}_{m,j}+v^{\rho,k-1,l}_{m,j}\right)(\tau,y) \right) \times\\
\\
\left( \delta v^{\rho,k,l}_{j,m}(\tau,y) \right) {\Big)}dy\\
\\
\mathbf{\delta v}^{\rho,k+1,l}(l-1,.)= 0,
\end{array}\right.
\end{equation}
and where $\delta v^{\rho,k+1,l}_{i}=v^{\rho,k+1,l}_{i}-v^{\rho,k,l}_{i}$ for $k\geq 1$, and 
$\delta v^{\rho,1,l}_j= v^{\rho,1,l}_j-v^{\rho,0,l}:=v^{\rho,1,l}_j-v^{\rho,l-1}_i(l-1,.)$. Note that $\delta v^{\rho,1,1}_j= v^{\rho,1,1}_j-h_j$ at the first time-step.
The equations for $v^{\rho,1,l}_i$ and for $\delta v^{\rho,k,l}_i$ are linearized and localized equations where by localisation we mean the fact that the global integral terms in the equation for $v^{\rho,1,l}_i$ and for $\delta v^{\rho,k,l}_i$ are given in terms of the initial data and the data from the previous iteration step respectively. Assuming that $v^{\rho,l-1}_i(l-1,.)\in H^2$ we show that $v^{\rho,1,l}(\tau,.)$ in $H^2$ uniformly with respect to $\tau$ and that the series $\left( \delta v^{\rho,k,l}_i(\tau,.)\right)$ satisfies a contraction property in $H^2$ uniformly with respect to $\tau\in [l-1,l]$.   We show that the functions $v^{\rho,1,l}_i$ are limits of a functional series where we have a contraction property for the elements of the series of form $\delta v^{\rho,k+1,l}$. This implies that the series (\ref{funciv}) evaluated a $\tau\in [l-1,l]$ converges in $H^2$ such that for $n=3$ 
we may apply the following extension of a standard Sobolev lemma.
\begin{lem}
For $s=\alpha+k+\frac{1}{2}n$ with $\alpha\in (0,1)$ we have
\begin{equation}
H^s\subset C^{\alpha},
\end{equation}
where $C^{\alpha}$ is the space of H\"{o}lder-continuous functions.
\end{lem}
We note that the $L^p$ estimates which are useful for higher dimensions may be used in the context of a more general lemma. We have
\begin{lem}
For $s>k+\frac{n}{p}$ with $\alpha\in (0,1)$ we have
\begin{equation}
H^{s,p}\subset C^{k},
\end{equation}
where $H^{s,p}$ is the space of functions $f$ where $\Lambda^sf\in L^p$ along with
\begin{equation}
\Lambda^s=\left[I-{2\pi}^{-2}\Delta\right]^{s/2}.
\end{equation}
\end{lem}
Applying such type of lemmas we can ensure that the first order coefficients (evaluated at time $\tau$)  of the equations which determine our approximations $v^{\rho,k,l}_i$ satisfy classical conditions which are sufficient for the existence of fundamental solutions of the associated linear parabolic equations of our scheme. Especially first order coefficients evaluated at $\tau$ are in the H\"{o}lder space $C^{\alpha}$ for some $\alpha\in (0,0.5)$ and  in $C^1_0$ uniformly with respect to $\tau$, i.e., the in the space of continuously differentiable functions which vanish at spatial infinity. 

The construction here defines a weak solution in $H^2$ and in $H^{2,\infty}$ as $n=3$. We have not mentioned $H^{2,\infty}$-estimates explicitly, but the estimated above can be adapted straightforwardly. Note that the application of the generalized Young inequality is even more simple in this case: for $r=\infty$ we have $1+\frac{1}{r}=\frac{1}{p}+\frac{1}{q}$, and we may use $p=q=2$ outside a ball and $p=\infty$ and $q=1$ inside a ball. Since  $v^{\rho,l}\in C^{\alpha}$ as the limit of the functional series in (\ref{funciv}), we have representations in terms of $v^{\rho,l}$ in terms  of the fundamental solution of
\begin{equation}
\frac{\partial p}{\partial \tau}-\rho_l\left( \sum_{j=1}^n \frac{\partial^2 p}{\partial x_j^2} 
-\sum_{j=1}^n v^{\rho,l}_j\frac{\partial  p}{\partial x_j}\right)=0,
\end{equation}
and this leads to the immediate conclusion that the solution is classical.
Note that for local restrictions to a bounded domain $\Omega\subset {\mathbb R}^n$ the series $|v^{\rho,k,l}_i(\tau,.)|_{\Omega}$ converges to a limit $ v^{\rho,l}(\tau,.)|_{\Omega}$ in a classical Banach space. Recall the following fact, which is better known for H\"{o}lder spaces.
\begin{prop}\label{propcl}
For open and bounded $\Omega\subset {\mathbb R}^n$ and consider the function space
\begin{equation}
\begin{array}{ll}
C^m\left(\Omega\right):={\Big \{} f:\Omega \rightarrow {\mathbb R}|~\partial^{\alpha}f \mbox{ exists~for~}~|\alpha|\leq m\\
\\
\mbox{ and }\partial^{\alpha}f \mbox{ has an continuous extension to } \overline{\Omega}{\Big \}}
\end{array}
\end{equation}
where $\alpha=(\alpha_1,\cdots ,\alpha_n)$ denotes a multiindex and $\partial^{\alpha}$ denote partial derivatives with respect to this multiindex. Then the function space $C^m\left(\overline{\Omega}\right)$ with the norm
\begin{equation}
|f|_m:=|f|_{C^m\left(\overline{\Omega}\right) }:=\sum_{|\alpha|\leq m}{\big |}\partial^{\alpha}f{\big |}
\end{equation}
is a Banach space. Here,
\begin{equation}
{\big |}f{\big |}:=\sup_{x\in \Omega}|f(x)|.
\end{equation}
\end{prop}
This leads to a second argument that the limit is indeed of form $v^{\rho,l}_i(\tau,.)\in H^2\cap C^2$ uniformly in $\tau$ and such that $\mathbf{v}^{\rho,l}$ satisfies the incompressible Navier-Stokes equation locally on $[l-1,l]\times {\mathbb R}^n$. This variation of argument has the advantage that it does not depend on dimension. On the other hand we can do the construction in $H^m$ instead of $H^2$ for $m>2+\frac{1}{2}n$, so this is a matter of taste. The last step then is to show that we have a linear bound of growth with respect to $H^2$. For this purpose it is essential to show that we have a global linear bound for the Leray projection term. It is at this point that it seems useful to introduce a control function as outlined in the introduction. We shall show that the $H^2$-contraction result can be extended to the controlled system, and then we shall show that 
\begin{equation}
|v^{r,\rho,l}_i(l,.)|_{H^2}\leq |v^{r,\rho,l-1}_i(l-1,.)|_{H^2}+C_2
\end{equation}
for a constant $C_2$ which is independent of the time step number $l$ and which holds for all $1\leq i\leq n$. This implies that we have a global bound
\begin{equation}
\max_{1\leq i\leq n}|v^{r,\rho,l}_i(l,.)|_{H^2}\leq \max_{1\leq i\leq n}|h_i|_{H^2}+lC_2,
\end{equation}
and since the control function $\mathbf{r}$ has a global linear bound as well this will show that
\begin{equation}
\max_{1\leq i\leq n}|v^{r,\rho,l}_i(l,.)|_{H^2}\leq \max_{1\leq i\leq n}|h_i|_{H^2}+lC^*_2,
\end{equation}
for some other constant $C^*_2$ which is independnet of the time step number $l$.
Now let us consider this program of proof in more detail. We start with the local contraction estimate for the uncontrolled system.
First we observe
\begin{lem}
Let $v^{\rho,l-1}_i(l-1,.)\in H^2\cap C^2$ for all $1\leq i\leq n$. Then there exists a classical solution $v^{\rho,1,l}_i$ of (\ref{scalparasystlin10v}) with
\begin{equation}
v^{\rho,1,l}(\tau,.)\in H^2\cap C^2
\end{equation}
for all $\tau\in [l-1,l]$. Moreover, for $n=3$ we have
\begin{equation}
v^{\rho,1,l}(\tau,.)\in H^{2,\infty}, 
\end{equation}
where the latter space denotes the Sobolev space with weak derivatives up to second order in $L^{\infty}$. 
\end{lem}
\begin{proof}
Since $v^{r,\rho,l-1}_i(l-1,.)\in H^2$ we have $v^{r,\rho,l-1}_i(l-1,.)\in H^2\cap C^{\alpha}$ for $\alpha \in (0,0.5)$. Hence the fundamental solution $p^l$ of
\begin{equation}
\frac{\partial p^l}{\partial \tau}-\rho_l\nu\sum_{j=1}^n \frac{\partial^2 p^l}{\partial x_j^2} 
+\rho_l\sum_{j=1}^n v^{\rho,l-1}_j(l-1,.)\frac{\partial p^l}{\partial x_j}=0
\end{equation}
 exists (constructible in the classical sense by the Levy expansion) and the solution of the Cauchy problem in (\ref{scalparasystlin10v}) has the representation
\begin{equation}\label{scalparasystlin10vproof}
\begin{array}{ll}
 v^{\rho,1,l}_i(\tau,x)=\int_{{\mathbb R}^n}v^{\rho,l-1}_i(l-1,y)p^l(\tau,x,l-1,y)dy\\
 \\
+\rho_l\int_{l-1}^{\tau}\int_{{\mathbb R}^n}\int_{{\mathbb R}^n}\sum_{j,k=1}^n \left( \frac{\partial v^{\rho,l-1}_j}{\partial x_k}\frac{\partial v^{\rho,l-1}_k}{\partial x_j}\right) (l-1,y)K_{n,i}(z-y)\times \\
\\
\times p^l(s,x,l-1,z)dzdyds.
\end{array}
\end{equation}
Here, recall that $K_{n,i}$ denotes the partial first order derivative of the kernel $K_n$ with respect to the $i$th variable. 
Hence we have $v^{\rho,1,l}_i(\tau,.)\in C^2$ for all $\tau\in [l-1,l]$ which follows from classical analysis of the Levy expansion of the fundamental solution where we may differentiate under the integral in order to get a representation for the derivatives of first order for $\tau>l-1$. Moreover, the second derivatives of the last integral in (\ref{scalparasystlin10vproof}) have an adjoint representation (cf. also the argument in \cite{KB1} and \cite{KNS}) such that the second derivatives of $v^{\rho,1,l}_i$ with respect to the spatial variables $x_k$ and $x_m$ is the sum of
\begin{equation}
\int_{{\mathbb R}^n}v^{\rho,l-1}_i(l-1,y)\frac{\partial^2}{\partial x_k\partial x_m}p^l(\tau,x,l-1,y)dy
\end{equation}
(which exists since $v^{\rho,l-1}_i(l-1,.)$ is H\"{o}lder), and the second summand
\begin{equation}\label{secsum}
\begin{array}{ll}
+\rho_l\int_{l-1}^{\tau}\int_{{\mathbb R}^n}\int_{{\mathbb R}^n}\sum_{p,j=1}^n \frac{\partial}{\partial x_i} \left( \frac{\partial v^{\rho,l-1}_p}{\partial x_j}\frac{\partial v^{\rho,l-1}_j}{\partial x_p}\right) (l-1,y)\frac{\partial}{\partial x_k}K_n(z-y)\times \\
\\
\times \frac{\partial}{\partial x_m}p^{l,*}(s,x,l-1,z)dzdyds
\end{array}
\end{equation}
(for the adjoint $p^{l,*}$ cf. also \cite{KB1} and \cite{KNS} ). The term
\begin{equation}
\int_{{\mathbb R}^n}\sum_{j,k=1}^n \frac{\partial}{\partial x_i}\left( \frac{\partial v^{\rho,l-1}_j}{\partial x_k}\frac{\partial v^{\rho,l-1}_k}{\partial x_j}\right) (l-1,y)\frac{\partial}{\partial x_k}K_n(z-y)dy
\end{equation}
corresponds to an $L^2$-function according to our lemma \ref{poisson} above, and Gaussian estimates for the first derivatives of the fundamental solution and its adjoint plus an application of the generalized Young inequality ensure that (\ref{secsum}) is in $L^2$ for each $\tau$ (first the integrand is in $L^2$ and then the integral up to $\tau$ is in $L^2$ where $\tau$ is considered as an parameter). Note that the Gaussian a priori estimate of the fundamental solution in lemma \ref{lemgauss} is $L^1\cap L^2$ for fixed $\tau>s$ as a function of $x-y$. Similar for the Gaussian a priori estimate in lemma \ref{lemgaussad}. Let us look at the second term more precisely since this is the term which defines the extension of our scheme for the multivariate Burgers equation. We have
\begin{equation}\label{secsum2}
\begin{array}{ll}
+\rho_l\int_{l-1}^{\tau}\int_{{\mathbb R}^n}\int_{{\mathbb R}^n}\frac{\partial}{\partial x_i}\sum_{i,j=1}^n \left( \frac{\partial v^{\rho,l-1}_i}{\partial x_j}\frac{\partial v^{\rho,l-1}_j}{\partial x_i}\right) (l-1,y)\frac{\partial}{\partial x_k}K_n(z-y)\times \\
\\
\times \frac{\partial}{\partial x_m}p^{l,*}(s,x,l-1,z)dzdyds\\
\\
\leq \rho_l\int_{l-1}^{\tau}\int_{{\mathbb R}^n}\int_{{\mathbb R}^n}{\Big |}\sum_{j,k=1}^n \frac{\partial}{\partial x_i}\left( \frac{\partial v^{\rho,l-1}_j}{\partial x_k}\frac{\partial v^{\rho,l-1}_k}{\partial x_j}\right) (l-1,y)\frac{\partial}{\partial x_k}K_n(z-y){\Big |}\times \\
\\
\times {\Big |}\frac{C}{(t-s)^{(n+1)/2}}\exp\left(-\frac{\lambda(x-y)^2}{4(t-s)} \right){\Big |}dzdyds .
\end{array}
\end{equation}
Now our lemma \ref{poisson} and the generalized Young inequality in (\ref{Y1}) and (\ref{Y2}) with $r=2$ and $p=2$ and $q=1$, i.e.,
\begin{equation}\label{Y1*}
 f\in L^2~\mbox {and}~ g\in L^1~\rightarrow  f\ast g\in L^2,\mbox{if}~1+\frac{1}{2}=1+\frac{1}{r},
\end{equation}
and
\begin{equation}\label{Y2*}
|f\ast g|_{L^2}\leq |f|_{L^2} |g|_{L^1}
\end{equation}
where $f$ corresponds to the Leray projection solution of the gradient of the pressure analyzed in lemma \ref{poisson} and $g$ corresponds to the Gaussian a priori bound where we observe that for $t>s$ we have
\begin{equation}
y\rightarrow \frac{C}{(t-s)^{(n+1)/2}}\exp\left(-\frac{\lambda(y)^2}{4(t-s)} \right)\in L^p.
\end{equation}
for $p\geq 1$.
 Similarly for the first term. Moreover, from the representation of both summands we observe that $v^{\rho,1,l}_i\in C^2$. We note that for $n=3$ we may apply lemma \ref{poisson} and the generalized Young inequality in (\ref{Y1}) and (\ref{Y2}) for $r=\infty$ and $p=q=\frac{1}{2}$ in order to get $L^{\infty}$ estimates. 
\end{proof}

Next in order to construct a local solution of the Navier-Stokes equation we establish a contaction property for the correction functionals $\delta v^{\rho,k,l}_i$ of the first linear approximation $v^{\rho,1,l}_i$ considered above.

\begin{lem}\label{contrlem} Let $n=3$
If $v^{\rho,l-1}_i(l-1,.)\in H^2\cap C^2$ for all $1\leq i\leq n$ then for some time ste size $\rho_l$ we have a contraction
\begin{equation}
\max_{i\in\left\lbrace 1,\cdots ,n\right\rbrace }\sup_{\tau\in [l-1,l]}|\delta v^{\rho,k,l}_i(\tau,.)|_{H^2}\leq \frac{1}{2}\max_{i\in\left\lbrace 1,\cdots ,n\right\rbrace }\sup_{\tau\in [l-1,l]}|\delta v^{\rho,k-1,l}_i(\tau,.)|_{H^2},
\end{equation}
where we denote
\begin{equation}
\max_{i\in\left\lbrace 1,\cdots,n\right\rbrace }\sup_{\tau\in [l-1,l]}|\delta v^{\rho,k,l}_i(\tau,.)|_{H^2}=\max_{i\in\left\lbrace 1,\cdots,n\right\rbrace }\sup_{\tau\in [l-1,l]}|\delta v^{\rho,k,l}_i(\tau,.)|_{H^2}.
\end{equation}
Moreover, for $n=3$ we also have the contraction estimate
\begin{equation}
|\delta v^{\rho,k,l}(\tau,.)|_{H^{2,\infty}}\leq \frac{1}{2}|\delta v^{\rho,k-1,l}(\tau,.)|_{H^{2,\infty}}.
\end{equation} 
\end{lem}

\begin{rem}
oreover, if $v^{\rho,l-1}_i(l-1,.)\in H^2\cap C^2$ this is a contraction in $|.|_{2}$ such that restrictions of the functional series $\left( v^{\rho,l,k}_i(\tau,.)\right)_{k}$ to an arbitrary bounded domain $\Omega\subset {\mathbb R}^n$ converge in a classical Banach space of twice differentiable functions with continuous extension at the boundary.
\end{rem}

\begin{proof} We prove the theorem in case $n=3$. 
\begin{equation}\label{serk}
v^{\rho,k,l}_j(\tau,.)=v^{\rho,1,l}_j(\tau,.)+\sum_{m=2}^k\delta v^{\rho,m,l}(\tau,.)\in H^2\subset C^{\alpha}
\end{equation}
for $\alpha\in (0,0.5)$ and uniformly with respect to $\tau\in [l-1,l]$. Moreover, we know inductively that $v^{\rho,k,l}_j(\tau,.)\in C^2$ for all $\tau\in [l-1,l]$. For $k=1$ (when we interpret the second summand in (\ref{serk}) to be zero) we know this from the previous lemma. Hence inductively with respect to the subiteration index $k$ we know that the fundamental solution $p^{k,l}$ of
\begin{equation}
\frac{\partial p^{k,l}}{\partial \tau}-\rho_l \nu\sum_{j=1}^n \frac{\partial^2 \delta p^{k,l}}{\partial x_j^2} 
+\rho_l\sum_{j=1}^n v^{\rho,k,l}_j\frac{\partial p^{k,l}}{\partial x_j}=0
\end{equation}
exists, and it follows that the solution of the linear problem (\ref{deltaurhok0}) has the representation
\begin{equation}\label{deltaurhok0proof}
\begin{array}{ll}
 \delta v^{\rho,k+1,l}_i(\tau,x)=\\
\\
-\rho_l\int_{l-1}^{\tau}\int_{{\mathbb R}^n}\sum_j\left(\delta v^{\rho,k,l}_j\frac{\partial v^{\rho,k-1,l}_i}{\partial x_j} \right)(s,y)p^{k,l}(\tau,x,s,y)dyds\\ 
\\
+\rho_l\int_{l-1}^{\tau}\int_{{\mathbb R}^n}\int_{{\mathbb R}^n}K_{n,i}(z-y){\Big (} \left( \sum_{j,k=1}^n\left( v^{\rho,k,l}_{k,j}+v^{\rho,k-1,l}_{k,j}\right)(\tau,y) \right) \times\\
\\
\left( \delta v^{\rho,k,l}_{j,k}(\tau,y) \right) {\Big)}p^{k,l}(\tau,x,s,z)dydsdydz.
\end{array}
\end{equation}
For the first order derivatives we have the representation
\begin{equation}\label{deltaurhok0proof2}
\begin{array}{ll}
 \frac{\partial}{\partial x_m}\delta v^{\rho,k+1,l}_i(\tau,x)=\\
\\
-\rho_l\int_{l-1}^{\tau}\int_{{\mathbb R}^n}\sum_j\left(\delta v^{\rho,k,l}_j\frac{\partial v^{\rho,k-1,l}_i}{\partial x_j} \right)(s,y)\frac{\partial}{\partial x_m}p^{k,l}(\tau,x,s,y)dyds\\ 
\\
+\rho_l\int_{l-1}^{\tau}\int_{{\mathbb R}^n}\int_{{\mathbb R}^n}K_{n,i}(z-y){\Big (} \left( \sum_{j,p=1}^n\left( v^{\rho,k,l}_{k,p}+v^{\rho,k-1,l}_{p,j}\right)(\tau,y) \right) \times\\
\\
\left( \delta v^{\rho,k,l}_{j,p}(\tau,y) \right) {\Big)}\frac{\partial}{\partial x_m}p^{k,l}(\tau,x,s,z)dydsdydz,
\end{array}
\end{equation}
and for the second order derivatives we have the representation
\begin{equation}\label{deltaurhok0proof3}
\begin{array}{ll}
 \frac{\partial^2}{\partial x_m\partial x_q}\delta v^{\rho,k+1,l}_i(\tau,x)=\\
\\
+\rho_l\int_{l-1}^{\tau}\int_{{\mathbb R}^n}\sum_j\frac{\partial}{\partial x_m}\left(\delta v^{\rho,k,l}_j\frac{\partial v^{\rho,k-1,l}}{\partial x_j} \right)(s,y)\frac{\partial}{\partial x_q}p^{k,l,*}(\tau,x,s,y)dyds\\ 
\\
+\rho_l\int_{l-1}^{\tau}\int_{{\mathbb R}^n}\int_{{\mathbb R}^n}K_{n,m}(z-y)\frac{\partial}{\partial x_i}{\Big (} \left( \sum_{j,p=1}^n\left( v^{\rho,k,l}_{p,j}+v^{\rho,k-1,l}_{p,j}\right)(\tau,y) \right) \times\\
\\
\left( \delta v^{\rho,k,l}_{j,p}(\tau,y) \right) {\Big)}\frac{\partial}{\partial x_q}p^{k,l,*}(\tau,x,s,z)dydsdydz,
\end{array}
\end{equation}
and where $p^{k,l,*}$ denotes the adjoint (consider also part I of this article). 
For $k=1$ we have the representation in (\ref{scalparasystlin10vproof}) such that 
\begin{equation}\label{scalparasystlin10vproof2}
\begin{array}{ll}
 v^{\rho,1,l}_i(\tau,x)-v^{\rho,0,l}(\tau,x)=\int_{{\mathbb R}^n}v^{\rho,l-1}_i(l-1,y)p^{0,l}(\tau,x,l-1,y)dy-v^{\rho,0,l}(\tau,x)\\
 \\
+\rho_l\int_{l-1}^{\tau}\int_{{\mathbb R}^n}\int_{{\mathbb R}^n}\sum_{i,j=1}^n \left( \frac{\partial v^{\rho,l-1}_i}{\partial x_j}\frac{\partial v^{\rho,l-1}_j}{\partial x_i}\right) (l-1,y)\frac{\partial}{\partial x_i}K_n(z-y)\times \\
\\
\times p^{0,l}(s,x,l-1,z)dzdyds.
\end{array}
\end{equation}
Well, we defined $v^{\rho,0,l}(\tau,x)=v^{\rho,l-1}(l-1,.)$. We could have defined $v^{\rho,0,l}_i$ such that the first summand on the right side of (\ref{scalparasystlin10vproof2}) cancels. Anyway classical analysis tells us hat there is a bound in the relevant norms. The essential term is the second summand which we may estimate using lemma \ref{poisson}. 
For the equation in (\ref{deltaurhok0proof3}) we get the following estimate (we provide more details of this estimate below). Note that we have for all $\tau >s$
\begin{equation}
\int_0^{\tau}{\Big|} \frac{C}{(\tau-s)^{(n+1)/2}}\exp\left(-\frac{\lambda(.)^2}{4(\tau-s)} \right){\Big |}_{L^1}ds\leq C
\end{equation}
for some constant $C>0$ wich is independent of $t-s$. This is due to the fact that locally we may use 
\begin{equation}
{\Big |} \frac{C}{(\tau-s)^{(n+1)/2}}\exp\left(-\frac{\lambda(.)^2}{4(\tau-s)} \right){\Big |}\leq \frac{c}{(t-s)^{\alpha}|x-y|^{n+1-2\alpha}}
\end{equation}
for some constant $c>0$ and some parameter $\alpha\in (0.5,1)$. In the complementary unbounded region it is clear that the time integral of the first spatial derivatives of the Gaussian has a uniform $L^1$-bound. First we observe that we may consider the spatial convolution first due to Fubini, and apply a Young inequality for fixed $t>s$. We get
\begin{equation}\label{deltaurhok0proof3}
\begin{array}{ll}
 {\Big |}\frac{\partial^2}{\partial x_m\partial x_q}\delta v^{\rho,k+1,l}_i(\tau,.){\Big |}_{L^2}\\
\\
\leq \rho_l\int_{l-1}^{\tau}{\Big |}\sum_j\frac{\partial}{\partial x_m}\left(\delta v^{\rho,k,l}_j\frac{\partial v^{\rho,k-1,l}_i}{\partial x_j}\right)(s,.){\Big |}_{L^2}{\Big |} \frac{C}{(\tau-s)^{(n+1)/2}}\exp\left(-\frac{\lambda(.)^2}{4(\tau-s)} \right){\Big |}_{L^1}ds\\ 
\\
+\rho_l\int_{l-1}^{\tau}{\Big |}\int_{{\mathbb R}^n}K_{n,m}(.-y){\Big (} \frac{\partial}{\partial x_i}\left( \sum_{j,p=1}^n\left( v^{\rho,k,l}_{p,j}+v^{\rho,k-1,l}_{p,j}\right)(s,y) \right) \times\\
\\
\left( \delta v^{\rho,k,l}_{j,p}(s,y) \right) {\Big)}dy{\Big |}_{L^2}{\Big|} \frac{C}{(\tau-s)^{(n+1)/2}}\exp\left(-\frac{\lambda(.)^2}{4(\tau-s)} \right){\Big |}_{L^1}ds
\end{array}
\end{equation}
Next we may use estimates for weighted $L^2$-products or an inductive assumption of boundedness of $\frac{\partial v^{\rho,k-1,l}_i}{\partial x_j}(s,.)$ and $\frac{\partial v^{\rho,k,l}_i}{\partial x_j}(s,.)$ in order to extract the function increments of form $\delta v^{\rho,k,l}_{j,p}$ and $\delta v^{\rho,k,l}_{j,p,k}$ in (\ref{deltaurhok0proof3}). We shall give more details of this estimate for weighted products below.
Another related method is to consider Sobolev product rules in $H^2$ for functions defined on ${\mathbb R}^3$, i.e. the rule that $|fg|_{H_2}\leq C_2|f|_{H^2}|g|_{H^2}$ for functions $f,g:{\mathbb R}^3\rightarrow {\mathbb R}$ along with $f,g\in H^2$. Note that all the terms
\begin{equation}
\begin{array}{ll}
\frac{\partial}{\partial x_i}\left( \sum_{j,p=1}^n\left( v^{\rho,k,l}_{p,j}+v^{\rho,k-1,l}_{p,j}\right)(s,y) \right)\left( \delta v^{\rho,k,l}_{j,p}(s,y) \right)\\
\\
=\left( \sum_{j,p=1}^n\left( v^{\rho,k,l}_{p,j,i}+v^{\rho,k-1,l}_{p,j,i}\right)(s,y)\left( \delta v^{\rho,k,l}_{j,p}(s,y) \right)\right) \\
\\
+\left( \sum_{j,p=1}^n\left( v^{\rho,k,l}_{p,j}+v^{\rho,k-1,l}_{p,j}\right)(s,y) \right)\left( \delta v^{\rho,k,l}_{j,p,i}(s,y) \right)
\end{array}
\end{equation}
appear in the classical $H^2$-Sobolev definition of ($s\in [l-1,l]$ fixed)
\begin{equation}
{\Big |}\sum_{j,p=1}^n\left( v^{\rho,k,l}_{p}+v^{\rho,k-1,l}_{p}\right)(s,.) \delta v^{\rho,k,l}_{j}(s,.) {\Big |}_{H^2}
\end{equation}
Hence we have 
\begin{equation}\label{deltaurhok0proof3prod}
\begin{array}{ll}
 {\Big |}\frac{\partial^2}{\partial x_m\partial x_q}\delta v^{\rho,k+1,l}_i(\tau,.){\Big |}_{L^2}\\
\\
\leq \rho_l\int_{l-1}^{\tau}{\Big |}\sum_j\left(\delta v^{\rho,k,l}_j v^{\rho,k-1,l}_i\right)(s,.){\Big |}_{H^2}{\Big |} \frac{C}{(\tau-s)^{(n+1)/2}}\exp\left(-\frac{\lambda(.)^2}{4(\tau-s)} \right){\Big |}_{L^1}ds\\ 
\\
+\rho_l\int_{l-1}^{\tau}{\Big |}\int_{{\mathbb R}^n}K_{n,m}(.-y){\Big (}  \sum_{j,p=1}^n\left( v^{\rho,k,l}_{p}+v^{\rho,k-1,l}_{p}\right)(s,y)  \times\\
\\
\delta v^{\rho,k,l}_{j}(s,y) {\Big)}dy{\Big |}_{H^2}{\Big|} \frac{C}{(\tau-s)^{(n+1)/2}}\exp\left(-\frac{\lambda(.)^2}{4(\tau-s)} \right){\Big |}_{L^1}ds\\
\\
\leq 
\rho_l\int_{l-1}^{\tau}C_2{\Big |}\sum_j\delta v^{\rho,k,l}_j(s,.){\Big |}_{H^2} 

{\Big |}v^{\rho,k-1,l}_i(s,.){\Big |}_{H^2}
{\Big |} \frac{C}{(\tau-s)^{(n+1)/2}}\exp\left(-\frac{\lambda(.)^2}{4(\tau-s)} \right){\Big |}_{L^1}ds\\ 
\\
+\rho_l\int_{l-1}^{\tau}{\Big |}\int_{{\mathbb R}^n}K_{n,m}(.-y){\Big (}  \sum_{j,p=1}^n\left( v^{\rho,k,l}_{p}+v^{\rho,k-1,l}_{p}\right)(s,y)  \times\\
\\
\delta v^{\rho,k,l}_{j}(s,y) {\Big)}dy{\Big |}_{H^2}{\Big|} \frac{C}{(\tau-s)^{(n+1)/2}}\exp\left(-\frac{\lambda(.)^2}{4(\tau-s)} \right){\Big |}_{L^1}ds
\end{array}
\end{equation}
In order to simplify the first term on the right side of (\ref{deltaurhok0proof3prod}) we may use the inductive assumption
\begin{equation}
{\Big |}v^{\rho,k-1,l}_i(s,.){\Big |}_{H^2}\leq C^l_{k-1}.
\end{equation}
In the global scheme the constants $C^l_k$ is a positive constant which depends linearly on $l$ but locally we have just a finite constant. There are several possibilities here, but all variations we have in mind use the fact that in dimension $n=3$ we have a localized Laplacian in $L^1$ which is in $L^2$ on the complementary domain (cf. Lemma \ref{poisson}, also for the definition of $\phi_{\epsilon}$ in the following).  
\begin{equation}\label{kernelesta}
\begin{array}{ll}
{\Big |}\int_{{\mathbb R}^n}K_{n,m}(.-y){\Big (}  \sum_{j,p=1}^n\left( v^{\rho,k,l}_{p}+v^{\rho,k-1,l}_{p}\right)(s,y)\delta v^{\rho,k,l}_{j}(s,y) {\Big)}dy{\Big |}_{H^2}\\
\\
={\Big |}\int_{{\mathbb R}^n}(\phi_{\epsilon}K_{n,m})(.-y){\Big (}  \sum_{j,p=1}^n\left( v^{\rho,k,l}_{p}+v^{\rho,k-1,l}_{p}\right)(s,y)\delta v^{\rho,k,l}_{j}(s,y) {\Big)}dy{\Big |}_{H^2}\\
\\
+{\Big |}\int_{{\mathbb R}^n}((1-\phi_{\epsilon})K_{n,m})(.-y){\Big (}  \sum_{j,p=1}^n\left( v^{\rho,k,l}_{p}+v^{\rho,k-1,l}_{p}\right)(s,y)\delta v^{\rho,k,l}_{j}(s,y) {\Big)}dy{\Big |}_{H^2}\\
\\
={\Big |}\int_{{\mathbb R}^n}(\phi_{\epsilon}K_{n,m})(y){\Big (}  \sum_{j,p=1}^n\left( v^{\rho,k,l}_{p}+v^{\rho,k-1,l}_{p}\right)(s,.-y)\delta v^{\rho,k,l}_{j}(s,.-y) {\Big)}dy{\Big |}_{H^2}\\
\\
+{\Big |}\int_{{\mathbb R}^n}((1-\phi_{\epsilon})K_{n,m})(y){\Big (}  \sum_{j,p=1}^n\left( v^{\rho,k,l}_{p}+v^{\rho,k-1,l}_{p}\right)(s,.-y)\delta v^{\rho,k,l}_{j}(s,.-y) {\Big)}dy{\Big |}_{H^2}
 \end{array}
\end{equation}
In the last step we used the convolution rule in order to make clear that the last to terms can be explictly written in terms of sums of $L^2$-norms where only the first order derivatives of the Laplacian kernel $K$ appear, i.e., we have representations by sums of $L^2$-norms such that the functions
\begin{equation}
((1-\phi_{\epsilon})K_{n,m})(.)\in L^2
\end{equation}
and
\begin{equation}\label{KnmL1}
(\phi_{\epsilon}K_{n,m})(.)\in L^1
\end{equation}
are untouched. As we said there a some variations of arguments possible now. Let us consider one first which exploits the $H^2$-product rule directly as far as possible. This method still uses the accompanying $H^{2,\infty}$-estimates which we shall consider below (and which are quite similar, in fact a little bit easier). There are other variations of argument that have the advantage that we can stay in $L^2$-theory and do not need additional but related arguments from $L^{\infty}$-theory. We shall consider these variations as well. 
For the second term on the right side of (\ref{deltaurhok0proof3prod}) we may use (\ref{KnmL1}) and the Young inequality, and writing the $|.|_{H^2}$ in Sobolev's classical sense as a sum of $L^2$ norms (the more modern sense may be the defininition via Fourier transforms) we obtain
\begin{equation}\label{ktrunc}
\begin{array}{ll}
{\Big |}\int_{{\mathbb R}^n}(\phi_{\epsilon}K_{n,m})(y){\Big (}  \sum_{j,p=1}^n\left( v^{\rho,k,l}_{p}+v^{\rho,k-1,l}_{p}\right)(s,.-y)\delta v^{\rho,k,l}_{j}(s,.-y) {\Big)}dy{\Big |}_{H^2}\\
\\
\leq C_K{\Big |}  \sum_{j,p=1}^n\left( v^{\rho,k,l}_{p}+v^{\rho,k-1,l}_{p}\right)(s,.)\delta v^{\rho,k,l}_{j}(s,.){\Big |}_{H^2}
\end{array}
\end{equation}
We may estimate the right side of (\ref{ktrunc}) using the $H^2$-product rule. We get
\begin{equation}\label{ktrunc}
\begin{array}{ll}
{\Big |}\int_{{\mathbb R}^n}(\phi_{\epsilon}K_{n,m})(y){\Big (}  \sum_{j,p=1}^n\left( v^{\rho,k,l}_{p}+v^{\rho,k-1,l}_{p}\right)(s,.-y)\delta v^{\rho,k,l}_{j}(s,.-y) {\Big)}dy{\Big |}_{H^2}\\
\\
\leq C_K2C^l_kC_22n^2\max_{j\in\left\lbrace 1,\cdots, n\right\rbrace }{\Big |}\delta v^{\rho,k,l}_{j}(s,.){\Big |}_{H^2}.
\end{array}
\end{equation}
For the last term on the right side of (\ref{deltaurhok0proof3prod}) note that with the function $\phi_{\epsilon}$ defined in Lemma \ref{poisson} above we have even
\begin{equation}\label{KH}
((1-\phi_{\epsilon})K_{n,m})(.)\in H^2,
\end{equation}
and we may introduce a constant $C_K$ as an upper bound, i.e.,
\begin{equation}
|((1-\phi_{\epsilon})K_{n,m})(.)|_{H^2}\leq C_K,
\end{equation}
Here we see that the matter is a more involved than in Lemma \ref{poisson} since we want to extract the functional increments. It makes sense to use Fourier transforms at this point which transform convolutions into products, but this cannot be done without caution, because we do not have that the function in (\ref{KH}) is in $L^1$ (or in $H^{2,1}$). However, we can give a weight to this function using inductive information that
\begin{equation}
 \sum_{j,p=1}^n\left( v^{\rho,k,l}_{p}+v^{\rho,k-1,l}_{p}\right)(s,.)\in H^2\cap C^2
\end{equation}
($C^2$ being the space of twice differentiable functions), and with an upper bound
\begin{equation}
{\Big |}  \max_{j,p\in \left\lbrace 1,\cdots n\right\rbrace} \left( v^{\rho,k,l}_{p}+v^{\rho,k-1,l}_{p}\right)(s,.){\Big |}_{H^2}\leq C_k+C_{k-1}\leq 2C_k
\end{equation}
we get 
\begin{equation}
\begin{array}{ll}
{\Big |}\int_{{\mathbb R}^n}((1-\phi_{\epsilon})K_{n,m})(y){\Big (}  \sum_{j,p=1}^n\left( v^{\rho,k,l}_{p}+v^{\rho,k-1,l}_{p}\right)(s,.-y)\delta v^{\rho,k,l}_{j}(s,.-y) {\Big)}dy{\Big |}_{H^2}\\
\\
\leq C_2C_K2C^l_kn^2\max_{j\in \left\lbrace 1,\cdots n\right\rbrace}{\Big |}  \delta v^{\rho,k,l}_{j}(s,.-y) {\Big)}dy{\Big |}_{H^2}
\end{array}
\end{equation}
We shall have a closer look at this argument below. In order to give an overview we work with a thick paint brush at the moment. But we shall give more details on the Fourier transform step below. Hence using these techniques we get for generic $C=C_2C_K2C^l_kn^2+C_K2C^l_kC_22n^2$
\begin{equation}\label{deltaurhok0proof4}
\begin{array}{ll}
 {\Big |}\frac{\partial^2}{\partial x_m\partial x_q}\delta v^{\rho,k+1,l}_i(\tau,.){\Big |}_{L^2}\\
\\
\leq \rho_l\int_{l-1}^{\tau}
C\max_{j,p\in \left\lbrace 1,\cdots ,n\right\rbrace }{\Big |}\delta v^{\rho,k,l}_{j,p}(s,.){\Big |}_{H^1}{\Big |} 
\frac{1}{(\tau-s)^{(n+1)/2}}\exp\left(-\frac{\lambda(.)^2}{4(\tau-s)} \right){\Big |}_{L^1}ds\\ 
\\
+\rho_l\int_{l-1}^{\tau}C\max_{j,p\in \left\lbrace 1,\cdots ,n\right\rbrace }| \delta v^{\rho,k,l}_{j,p}(s,.){\Big |}_{H^1}{\Big|} \frac{1}{(\tau-s)^{(n+1)/2}}\exp\left(-\frac{\lambda(.)^2}{4(\tau-s)} \right){\Big |}_{L^1}ds.
\end{array}
\end{equation}
Next we can take suprema and apply time intergability of $L^1$ norms of first derivatives of the Gaussian. We get
\begin{equation}\label{deltaurhok0proof5}
\begin{array}{ll}
 {\Big |}\frac{\partial^2}{\partial x_m\partial x_q}\delta v^{\rho,k+1,l}_i(\tau,.){\Big |}_{L^2}\\
\\
\leq \rho_l
C'C\left( \max_{j\in \left\lbrace 1,\cdots ,n\right\rbrace }\sup_{s\in [l-1,l]}{\Big |}\delta v^{\rho,k,l}_{j}(s,.){\Big |}_{H^1}\right) \\ 
\\
+\rho_lC'C\max_{j,p\in \left\lbrace 1,\cdots ,n\right\rbrace }\sup_{s\in [l-1,l]}| \delta v^{\rho,k,l}_{j,p}(s,.){\Big |}_{H^1}\\
\\
\leq \rho_l
C'C\max_{j\in \left\lbrace 1,\cdots ,n\right\rbrace }\sup_{s\in [l-1,l]}{\Big |}\delta v^{\rho,k,l}_j(s,.){\Big |}_{H^2}\\ 
\end{array}
\end{equation}
Let us summarize and have a closer look at this estimate where we take more attention to the Fourier transform part. For the second order derivatives we get products of the form
\begin{equation}
\left( \sum_{j,p=1}^n\left( v^{\rho,k,l}_{p,j,i}+v^{\rho,k-1,l}_{p,j,i}\right)(\tau,y) \right)
\left( \delta v^{\rho,k,l}_{j,p}(\tau,y) \right)
\end{equation}
and products of the form
\begin{equation}
\left( \sum_{j,p=1}^n\left( v^{\rho,k,l}_{p,j}+v^{\rho,k-1,l}_{p,j}\right)(\tau,y) \right)
\frac{\partial}{\partial x_i}\left( \delta v^{\rho,k,l}_{j,p}(\tau,y) \right)
\end{equation}
in the Leray projection term.
In order to estimate the latter Leray projection term (as an example) we  may split up the Laplacian kernel and use the estimate 
\begin{equation}
\begin{array}{ll}
{\Big |}\int_{{\mathbb R}^n}K_{n,m}(z-y){\Big (} \left( \sum_{j,p=1}^n\left( v^{\rho,k,l}_{p,j}+v^{\rho,k-1,l}_{p,j}\right)(\tau,y) \right) \times\\
\\
\frac{\partial}{\partial x_i}\left( \delta v^{\rho,k,l}_{j,p}(\tau,y) \right) {\Big)}dy{\Big |}_{L^2} C'\\
\\
\leq {\Big |}\int_{{\mathbb R}^n}\phi_{\epsilon}K_{n,m}(z-y){\Big (} \left( \sum_{j,p=1}^n\left( v^{\rho,k,l}_{p,j}+v^{\rho,k-1,l}_{p,j}\right)(\tau,y) \right) \times\\
\\
\frac{\partial}{\partial x_i}\left( \delta v^{\rho,k,l}_{j,p}(\tau,y) \right) {\Big)}dy{\Big |}_{L^2} C'\\
\\
+{\Big |}\int_{{\mathbb R}^n}(1-\phi_{\epsilon})K_{n,m}(z-y){\Big (} \left( \sum_{j,p=1}^n\left( v^{\rho,k,l}_{p,j}+v^{\rho,k-1,l}_{p,j}\right)(\tau,y) \right) \times\\
\\
\frac{\partial}{\partial x_i}\left( \delta v^{\rho,k,l}_{j,p}(\tau,y) \right) {\Big)}dy{\Big |}_{L^2} C'
\end{array}
\end{equation}
For the first term on the right side of the latter inequality we may use
\begin{equation}
{\big |}\phi_{\epsilon}K_{,i}{\big |}_{L^1}\leq C
\end{equation}
for some constant $C>0$, the product rule in $H^2$, and the generalized Young inequality. Estimation of the second term on the right side of the latter inequality we proceed is trickier.
However the estimates for second order derivatives are in fact easier. We can shift derivatives from the fundamental solution (using the adjoint) and estimate the term
\begin{equation}
\begin{array}{ll}
{\Big |}\int_{{\mathbb R}^n}\left( (1-\phi_{\epsilon})K_{n,i}(z-y)\right)_{,q,m} {\Big (} \left( \sum_{j,p=1}^n\left( v^{\rho,k,l}_{p,j}+v^{\rho,k-1,l}_{p,j}\right)(\tau,y) \right) \times\\
\\
\left( \delta v^{\rho,k,l}_{j,p}(\tau,y) \right) {\Big)}dy{\Big |}_{L^2},
\end{array}
\end{equation}
and this can be done along the lines discussed before since
\begin{equation}\label{thirdder}
\begin{array}{ll}
\left( (1-\phi_{\epsilon}(z-y))K_{n,i}(z-y)\right)_{,q,m}=\\
\\
 (1-\phi_{\epsilon}(z-y))K_{n,i,q,m}(z-y)+\\
\\
(-\phi_{\epsilon})_{,q}K_{n,i,m}(z-y)+(-\phi_{\epsilon})_{,q,m}K_{n,i}(z-y)
\end{array}
\end{equation}
The third derivatives of $K_n$ are in $L^1$, hence all summands on the right side of equation (\ref{thirdder}) are indeed in $L^1$, and we can proceed as before and use the generalized Young inequality and the $H^2$- product rule (in case of dimension $n=3$). This is sufficient for second derivatve terms as we estimated in (\ref{deltaurhok0proof3}). However, if we want to estimate terms of the form 
\begin{equation}
{\Big |}\frac{\partial}{\partial x_m}\delta v^{\rho,k+1,l}_i(\tau,.){\Big |}_{L^2},
\end{equation}
or of the form 
\begin{equation}
{\Big |}\delta v^{\rho,k+1,l}_i(\tau,.){\Big |}_{L^2},
\end{equation}
then the last argument is not sufficient since in the calculations there always appears a term  $(1-\phi_{\epsilon}(.))K_{n,i,q}(.)$ or $(1-\phi_{\epsilon}(.))K_{n,i}(.)$, and such terms are not in $L^1$. However, we can adapt this simple argument to $H^1$-estimates and $L^2$-estimate. Consider first order derivatives first. In this case we have to deal with expressions of the form
\begin{equation}\label{firstorderunc}
\begin{array}{ll}
{\Big |}\int_{{\mathbb R}^n}\left( (1-\phi_{\epsilon})K_{n,i}(z-y)\right)_{,q} {\Big (} \left( \sum_{j,p=1}^n\left( v^{\rho,k,l}_{p,j}+v^{\rho,k-1,l}_{p,j}\right)(\tau,y) \right) \times\\
\\
\left( \delta v^{\rho,k,l}_{j,p}(\tau,y) \right) {\Big)}dy{\Big |}_{L^2},
\end{array}
\end{equation}
where the derivative indexed by $q$ is from the first order derivative of the fundamental solution (or its adjoint). Now observe that
\begin{equation}\label{productrep}
\begin{array}{ll}
 \sum_{j,p=1}^n\left( v^{\rho,k,l}_{p,j}+v^{\rho,k-1,l}_{p,j}\right)(\tau,y) 
\left( \delta v^{\rho,k,l}_{j,p}(\tau,y) \right) \\
\\
=\left( \sum_{j,p=1}^n\left( v^{\rho,k,l}_{p}+v^{\rho,k-1,l}_{p}\right)(\tau,y) \right)
\left( \delta v^{\rho,k,l}_{j,p}(\tau,y) \right)_{,j}\\
\\
- \sum_{j,p=1}^n\left( v^{\rho,k,l}_{p}+v^{\rho,k-1,l}_{p}\right)(\tau,y) 
\left( \delta v^{\rho,k,l}_{j,p,j}(\tau,y)\right). 
\end{array}
\end{equation}
Concerning the first term on the right side of (\ref{productrep}) we may shift the derivative indexed by $j$ to the function $\left( (1-\phi_{\epsilon})K_{n,i}(.)\right)_{,q} $,
 and for the second term we use an inductively assumed  $H^{2,\infty}$ upper bound constant $2C^l_k$ of $\left( v^{\rho,k,l}_{p}+v^{\rho,k-1,l}_{p}\right)$, 
 and then shift one derivative of $ \delta v^{\rho,k,l}_{j,p,j}$ again to the function $\left( (1-\phi_{\epsilon})K_{n,i}(.)\right)_{,q} $
Hence, for (\ref{firstorderunc}) we have the upper bound
\begin{equation}\label{H1est}
\begin{array}{ll}
{\Big |}\int_{{\mathbb R}^n}\left( (1-\phi_{\epsilon})K_{n,i}(z-y)\right)_{,q} {\Big (} \left( \sum_{j,p=1}^n\left( v^{\rho,k,l}_{p,j}+v^{\rho,k-1,l}_{p,j}\right)(\tau,y) \right) \times\\
\\
\left( \delta v^{\rho,k,l}_{j,p}(\tau,y) \right) {\Big)}dy{\Big |}_{L^2}\\
\\
\leq {\Big |}\int_{{\mathbb R}^n}
\left( (1-\phi_{\epsilon})K_{n,i}(z-y)\right)_{,q,j} {\Big (}
 \left( \sum_{j,p=1}^n\left( v^{\rho,k,l}_{p}+v^{\rho,k-1,l}_{p}\right)(\tau,y) \right) 
\times\\
\\
\left( \delta v^{\rho,k,l}_{j,p}(\tau,y) \right) {\Big)}dy{\Big |}_{L^2}\\
\\
+{\Big |}\int_{{\mathbb R}^n}\left( (1-\phi_{\epsilon})K_{n,i}(z-y)\right)_{,q} {\Big (} \left( \sum_{j,p=1}^n\left( v^{\rho,k,l}_{p}+v^{\rho,k-1,l}_{p}\right)(\tau,y) \right) \times\\
\\
 \left( \delta v^{\rho,k,l}_{j,p,j}(\tau,y) \right) {\Big)}dy{\Big |}_{L^2}\\
 \\
 \leq {\Big |}\int_{{\mathbb R}^n}\left( (1-\phi_{\epsilon})K_{n,i}(z-y)\right)_{,q,j} {\Big (} \left( \sum_{j,p=1}^n\left( v^{\rho,k,l}_{p}+v^{\rho,k-1,l}_{p}\right)(\tau,y) \right) \times\\
 \\
 \left( \delta v^{\rho,k,l}_{j,p}(\tau,y) \right) {\Big)}dy{\Big |}_{L^2}\\
 \\
 +{\Big |}\int_{{\mathbb R}^n}\left( (1-\phi_{\epsilon})K_{n,i}(z-y)\right)_{,q,j} 2 n^2 C^l_k 
 \left( \delta v^{\rho,k,l}_{j,p}(\tau,y) \right) {\Big)}dy{\Big |}_{L^2}
\end{array}
\end{equation}
Then we can apply the argument above. Do we have enough derivatives to do the $L^2$ estimates by the same method ? Yes, we have. In this case we do not gain a derivative from the fundamental solution (or its adjoint), but look at the equation (\ref{productrep}) again. Using the inductively assumed $H^{2,\infty}$-upper bound (the correctness of the inductive assumption will be shown below), for the right side of (\ref{productrep}) we have the upper bound
\begin{equation}\label{productrep2}
\begin{array}{ll}
 |\sum_{j,p=1}^n\left( v^{\rho,k,l}_{p,j}+v^{\rho,k-1,l}_{p,j}\right)(\tau,y) 
\left( \delta v^{\rho,k,l}_{j,p}(\tau,y) \right)| \\
\\
\leq  \sum_{j,p=1}^n2C^l_k|
 \delta v^{\rho,k,l}_{j,p,j}(\tau,y) |
+| \sum_{j,p=1}^n2 C^l_k
\left( \delta v^{\rho,k,l}_{j,p,j}(\tau,y)\right)|, 
\end{array}
\end{equation}
and then we may shift both derivatives to the function
$\left( (1-\phi_{\epsilon})K_{n,i}(z-y)\right)$ and apply the same argument. Hence, this varaition of argument also holds for the $L^2$-estimates too.

There is an other method to deal  with these terms and in order to show this we have to go deeper into the $H^2$-product rule.
First for small $\delta>0$ we define
\begin{equation}
\psi_{\epsilon}(z):=\exp\left(-\delta z^2 \right),
\end{equation}
such that for $\delta >0$
\begin{equation}
\psi_{\delta}(.)(1-\phi_{\epsilon}(.))K_{n,i}(.)\in L^1,
\end{equation}
and
\begin{equation}
\lim_{\delta \downarrow 0}\psi_{\delta}(.)(1-\phi_{\epsilon}(.))K_{n,i}(.)=(1-\phi_{\epsilon}(.))K_{n,i}(.).
\end{equation}
Especially, we note that for $\delta >0$ the Fourier transform of    $\psi_{\delta}(.)(1-\phi_{\epsilon}(.))K_{n,i}(.)$ exists.
 We have to study the limit (as $\delta \downarrow 0$) of
\begin{equation}\label{plancherel}
\begin{array}{ll}
{\Big |}\int_{{\mathbb R}^n}\left( \psi_{\delta}(1-\phi_{\epsilon})K_{n,i}(z-y)\right) {\Big (} \left( \sum_{j,p=1}^n\left( v^{\rho,k,l}_{p,j}+v^{\rho,k-1,l}_{p,j}\right)(\tau,y) \right) \times\\
\\
\left( \delta v^{\rho,k,l}_{j,p}(\tau,y) \right) {\Big)}dy{\Big |}^2_{L^2}\\
\\
={\Big |}{\cal F}\left(\psi_{\delta} (1-\phi_{\epsilon})K_{n,i}(.)\right){\Big |}^2_{L^2} {\Big |}{\cal F}\left( \sum_{j,p=1}^n\left( v^{\rho,k,l}_{p,j}+v^{\rho,k-1,l}_{p,j}\right)(\tau,.) \right) \times\\
\\
\left( \delta v^{\rho,k,l}_{j,p}(\tau,.) \right) {\Big)}{\Big |}^2_{L^2}\\
\\
={\Big |}{\cal F}\psi_{\delta}\left( (1-\phi_{\epsilon})K_{n,i}(.)\right){\Big |}^2_{L^2} {\Big |}\left( \sum_{j,p=1}^n\left( v^{\rho,k,l}_{p,j}+v^{\rho,k-1,l}_{p,j}\right)(\tau,.) \right) \times\\
\\
\left( \delta v^{\rho,k,l}_{j,p}(\tau,.) \right) {\Big)}{\Big |}^2_{L^2}
\end{array}
\end{equation}
where we know that $\left( (1-\phi_{\epsilon})K_{n,i}(.)\right)_{,m}\in L^2$, and  $\delta v^{\rho,k,l}_{j,p}(\tau,.)\in L^2$, and  $\left( v^{\rho,k,l}_{p,j}+v^{\rho,k-1,l}_{p,j}\right)(\tau,.)\in L^2$ such that we can use the Plancerel formula (${\cal F}$ denotes the operator of Fourier transformation and ${\cal F}$ ist inverse). Next for the term
\begin{equation}
\begin{array}{ll}
 {\Big |}\left( \sum_{j,p=1}^n\left( v^{\rho,k,l}_{p,j}+v^{\rho,k-1,l}_{p,j}\right)(\tau,.) \right)
\left( \delta v^{\rho,k,l}_{j,p}(\tau,.) \right) {\Big)}{\Big |}^2_{L^2}\\
\\
={\Big |}{\cal F}^{-1}{\cal F}\left( \sum_{j,p=1}^n\left( v^{\rho,k,l}_{p,j}+v^{\rho,k-1,l}_{p,j}\right)(\tau,.) \right)
{\cal F}^{-1}{\cal F}\left( \delta v^{\rho,k,l}_{j,p}(\tau,.) \right) {\Big)}{\Big |}^2_{L^2}\\
\\
{\Big |}{\cal F}^{-1}\left( {\cal F}\left( \sum_{j,p=1}^n\left( v^{\rho,k,l}_{p,j}+v^{\rho,k-1,l}_{p,j}\right)(\tau,.) \right)
\star{\cal F}\left( \delta v^{\rho,k,l}_{j,p}(\tau,.) \right) {\Big)}\right) {\Big |}^2_{L^2}\\
\\
={\Big |}{\cal F}\left( \sum_{j,p=1}^n\left( v^{\rho,k,l}_{p,j}+v^{\rho,k-1,l}_{p,j}\right)(\tau,.) \right)
\star{\cal F}\left( \delta v^{\rho,k,l}_{j,p}(\tau,.) \right) {\Big)}{\Big |}^2_{L^2},
\end{array}
\end{equation}
where $\star$ denotes convolution.
Now, inductively $v^{\rho,k,l}_{p,j}+v^{\rho,k-1,l}_{p,j}$ and  $\delta v^{\rho,k,l}_{j,p}(\tau,.)$ are in $H^1$, hence we have for $s\leq 1$
\begin{equation}
\begin{array}{ll}
{\Big |}\frac{1}{\left(1+|.|^2\right)^s }{\cal F}\left( \sum_{j,p=1}^n\left( v^{\rho,k,l}_{p,j}+v^{\rho,k-1,l}_{p,j}\right)(\tau,.) \right)
\star (1+|.|^2)^s{\cal F}\left( \delta v^{\rho,k,l}_{j,p}(\tau,.) \right) {\Big)}{\Big |}^2_{L^2}\\
\\
\leq C_s^2{\Big |}{\cal F}\left( \sum_{j,p=1}^n\left( v^{\rho,k,l}_{p,j}+v^{\rho,k-1,l}_{p,j}\right)(\tau,.) \right){\Big |}^2_{L^2}
{\Big |}(1+|.|^2)^s{\cal F}\left( \delta v^{\rho,k,l}_{j,p}(\tau,.) \right) {\Big)}{\Big |}^2_{L^2},
\end{array}
\end{equation}
where we note that for
\begin{equation}
u(\xi)=\int(1+|y|^2)^{-s/2}v(\xi-y)w(y)dy
\end{equation}
and $v,w\in L^2$ we have
\begin{equation}
|u|_{L^2}\leq C_s|v|_{L^2}|w|_{L^2}.
\end{equation}
Hence,
\begin{equation}
\begin{array}{ll}
{\Big |}\frac{1}{\left(1+|.|^2\right)^s }{\cal F}\left( \sum_{j,p=1}^n\left( v^{\rho,k,l}_{p,j}+v^{\rho,k-1,l}_{p,j}\right)(\tau,.) \right)
\star (1+|.|^2)^s{\cal F}\left( \delta v^{\rho,k,l}_{j,p}(\tau,.) \right) {\Big)}{\Big |}_{L^2}\\
\\
\leq C_s{\Big |}{\cal F}\left( \sum_{p=1}^n\left( v^{\rho,k,l}_{p}+v^{\rho,k-1,l}_{p}\right)(\tau,.) \right){\Big |}_{H^1}
{\Big |} \delta v^{\rho,k,l}(\tau,.) {\Big)}{\Big |}_{H^2}.
\end{array}
\end{equation}
Summarizing, we have
\begin{equation}
\begin{array}{ll}
{\Big |}\int_{{\mathbb R}^n}\left( \psi_{\delta}(1-\phi_{\epsilon})K_{n,i}(z-y)\right) {\Big (} \left( \sum_{j,p=1}^n\left( v^{\rho,k,l}_{p,j}+v^{\rho,k-1,l}_{p,j}\right)(\tau,y) \right) \times\\
\\
\left( \delta v^{\rho,k,l}_{j,p}(\tau,y) \right) {\Big)}dy{\Big |}_{L^2}\\
\\
\leq {\Big |}{\cal F}\psi_{\delta}\left( (1-\phi_{\epsilon})K_{n,i}(.)\right){\Big |}_{L^2} {\Big |}\left( \sum_{j,p=1}^n\left( v^{\rho,k,l}_{p}+v^{\rho,k-1,l}_{p}\right)(\tau,.) \right){\Big |}_{H^1} \times\\
\\
{\Big |}\left( \delta v^{\rho,k,l}_{j,p}(\tau,.) \right) {\Big)}{\Big |}_{H^2}
\end{array}
\end{equation}
This leads to
\begin{equation}\label{last}
\begin{array}{ll}
{\Big |}\int_{{\mathbb R}^n}\left( (1-\phi_{\epsilon})K_{n,i}(z-y)\right) {\Big (} \left( \sum_{j,p=1}^n\left( v^{\rho,k,l}_{p,j}+v^{\rho,k-1,l}_{p,j}\right)(\tau,y) \right) \times\\
\\
\left( \delta v^{\rho,k,l}_{j,p}(\tau,y) \right) {\Big)}dy{\Big |}_{L^2}\\
\\
\leq {\Big |}\left( (1-\phi_{\epsilon})K_{n,i}(.)\right){\Big |}_{L^2} {\Big |}\left( \sum_{j,p=1}^n\left( v^{\rho,k,l}_{p}+v^{\rho,k-1,l}_{p}\right)(\tau,.) \right){\Big |}_{H^2} \times\\
\\
{\Big |}\left( \delta v^{\rho,k,l}_{j,p}(\tau,.) \right) {\Big)}{\Big |}_{H^2}
\leq C{\Big |}\left( \delta v^{\rho,k,l}_{j,p}(\tau,.) \right) {\Big)}{\Big |}_{H^2}
\end{array}
\end{equation}
where we may use that, we have a $L^2$-bound $\sup_{\delta >0}{\Big |}\psi_{\delta}\left( (1-\phi_{\epsilon})K_{n,i}(.)\right){\Big |}_{L^2}\leq C<\infty$ for some $C>0$, i.e., a  bound independent of $\delta >0$. Note that we have some freedom here since we replaced ${\Big |}\left( \sum_{j,p=1}^n\left( v^{\rho,k,l}_{p}+v^{\rho,k-1,l}_{p}\right)(\tau,.) \right){\Big |}_{H^1}$ by the stonger norm ${\Big |}\left( \sum_{j,p=1}^n\left( v^{\rho,k,l}_{p}+v^{\rho,k-1,l}_{p}\right)(\tau,.) \right){\Big |}_{H^2}$. So an alternative method is to examine the product of the right side of \ref{last} more closely and use the fact that we can factor out a weight $(1+|\xi|)^{-1}$ since the second factor is in $H^1$.

We still have to show that the inductive assumption concerning the $H^{2,\infty}$ estimates is justified (at least for some variations of arguments we made excessive use of this hypothesis), but let us assume for a moment that this has been shown in order to close the argument.
We have obtained
\begin{equation}\label{deltaurhok0proof5again}
\begin{array}{ll}
 {\Big |}\frac{\partial^2}{\partial x_m\partial x_q}\delta v^{\rho,k+1,l}_i(\tau,.){\Big |}_{L^2}\\
\\
\leq \rho_l
2C'C\max_{j\in \left\lbrace 1,\cdots ,n\right\rbrace }\sup_{s\in [l-1,l]}{\Big |}\delta v^{\rho,k,l}_j(s,.){\Big |}_{H^2}\\ 
\end{array}
\end{equation}
where $C=C_2C_K2C^l_kn^2+C_K2C^l_kC_22n^2$ for all $j,m,q\in \left\lbrace 1,\cdots ,n\right\rbrace$. Hence,
\begin{equation} 
\begin{array}{ll}
 \sum_{p,m=1}^n\max_{i\in \left\lbrace 1,\cdots ,n \right\rbrace }\sup_{\tau\in [l-1,l]}{\Big |}\frac{\partial^2}{\partial x_m\partial x_q}\delta v^{\rho,k+1,l}_i(\tau,.){\Big |}_{L^2}\\
\\
\leq \rho_l
2n^2C'C\max_{j\in \left\lbrace 1,\cdots ,n\right\rbrace }\sup_{s\in [l-1,l]}{\Big |}\delta v^{\rho,k,l}_j(s,.){\Big |}_{H^2}\\ 
\end{array}
\end{equation}
Assuming w.l.o.g. that all constants $C_2,C_K,C^l_k,C'\geq 1$ and that the upper bound constant $C'$ is also an upper bound for the integrated Gaussian itself (we assumed it to be an upper bound for the time-integrated first order spatial derivative of the Gaussian) we have an estimate with the same constants $\delta v^{\rho,k+1,l}_i(\tau,.)$ and its first spatial derivatives. The number of terms $n^2+n+1$ is bounded by $(n+1)^2$, hence we surely have 
\begin{equation} 
\begin{array}{ll}
 \max_{i\in \left\lbrace 1,\cdots ,n \right\rbrace }\sup_{\tau\in [l-1,l]}{\Big |}\delta v^{\rho,k+1,l}_i(\tau,.){\Big |}_{H^2}\\
\\
\leq \rho_l
2(n+1)^2C'C\max_{j\in \left\lbrace 1,\cdots ,n\right\rbrace }\sup_{s\in [l-1,l]}{\Big |}\delta v^{\rho,k,l}_j(s,.){\Big |}_{H^2} 
\end{array}
\end{equation}
Since the increment ${\Big |} \delta v^{\rho,1,l}{\Big |}_{H^2}$ can be assumed to be smaller the $\frac{1}{2}$ the constant $C^l:=2C^l_1$ is an upper bound for all constants $C^l_k$.
Hence we may choose
\begin{equation}
\rho_l\leq \frac{1}{4(n+1)^2C'\left( C_2C_K2C^ln^2+C_K2C^l_kC_22n^2\right) },
\end{equation}
and with this choice we get
\begin{equation}
\max_{j\in \left\lbrace 1,\cdots ,n\right\rbrace }\sup_{\tau\in [l-1,l]}{\Big |}\delta v^{\rho,k+1,l}_i(\tau,.){\Big |}_{H^2}\leq \frac{1}{2}\max_{j\in \left\lbrace 1,\cdots ,n\right\rbrace }\sup_{\tau\in [l-1,l]}
{\Big |} \delta v^{\rho,k,l}{\Big |}_{H^2}. 
\end{equation}
It is clear that choosing $\rho_l$ small if necessary we can put any positive bound on ${\Big |} \delta v^{\rho,1,l}{\Big |}_{H^2}$ in the first step.
Finally, we have to show that the inductive upper $H^{2,\infty}$-bound is correct. However, this is much simpler then the $H^2$ estimates. Consider 
(\ref{deltaurhok0proof3}) in $L^{\infty}$. For simplicity let $G^k_{l,q}$ denote a Gaussian majoriant of $\frac{\partial}{\partial x_q}p^{k,l,*}$ in (\ref{deltaurhok0proof3}). Fixing time (Fubini) and applying a generalized form of Young's inequality with $r=\infty$ and where we measure the Gaussian for fixed $t>s$ in $L^1$. Then we have 
\begin{equation}\label{deltaurhok0proof3againandagain}
\begin{array}{ll}
 \frac{\partial^2}{\partial x_m\partial x_q}\delta v^{\rho,k+1,l}_i(\tau,x)=\\
\\
+\rho_l\int_{l-1}^{\tau}{\Big |}
\sum_j\frac{\partial}{\partial x_m}
\left(\delta v^{\rho,k,l}_j\frac{\partial v^{\rho,k-1,l}}{\partial x_j} \right)(s,.){\Big |}_{L^{\infty}}{\Big |}G^k_{l,q}(t-s,.){\Big |}ds\\ 
\\
+\rho_l\int_{l-1}^{\tau}{\Big|}\int_{{\mathbb R}^n}K_{n,m}(.-y)\frac{\partial}{\partial x_i}{\Big (} \left( \sum_{j,p=1}^n\left( v^{\rho,k,l}_{p,j}+v^{\rho,k-1,l}_{p,j}\right)(\tau,y) \right) \times\\
\\
\left( \delta v^{\rho,k,l}_{j,p}(\tau,y) \right) {\Big)}dy{\Big |}_{L^{\infty}}{\Big |}G^k_{l,q}(t-s,.){\Big |}_{L^1},
\end{array}
\end{equation}
This means that we have extract the function increments $\delta v^{\rho,k,l}_{j,p}$ from the expressions
\begin{equation}\label{linfty1}
{\Big |}
\sum_j\frac{\partial}{\partial x_m}
\left(\delta v^{\rho,k,l}_j\frac{\partial v^{\rho,k-1,l}_i}{\partial x_j} \right)(s,.){\Big |}_{L^{\infty}},
\end{equation}
and from the expression
\begin{equation}
\begin{array}{ll}
{\Big|}\int_{{\mathbb R}^n}K_{n,m}(.-y)\frac{\partial}{\partial x_i}{\Big (} \left( \sum_{j,p=1}^n\left( v^{\rho,k,l}_{p,j}+v^{\rho,k-1,l}_{p,j}\right)(\tau,y) \right) \times\\
\\
\left( \delta v^{\rho,k,l}_{j,p}(\tau,y) \right) {\Big)}dy{\Big |}_{L^{\infty}}.
\end{array}
\end{equation}
For the first term (\ref{linfty1}) we have
\begin{equation}\label{linfty1}
\begin{array}{ll}
{\Big |}
\sum_j\frac{\partial}{\partial x_m}
\left(\delta v^{\rho,k,l}_j\frac{\partial v^{\rho,k-1,l}_i}{\partial x_j} \right)(s,.){\Big |}_{L^{\infty}}\\
\\
\leq {\Big |}
\sum_j\frac{\partial v^{\rho,k-1,l}_i}{\partial x_j} (s,.){\Big |}_{L^{\infty}}
\max{j\in \left\lbrace 1,\cdots, n\right\rbrace }{\Big |}\delta v^{\rho,k,l}_j(s,.){\Big |}_{H^{1,\infty}}\\
\\
\leq n {\Big |}
\max{i,j\in \left\lbrace 1,\cdots, n\right\rbrace }\frac{\partial v^{\rho,k-1,l}_i}{\partial x_j} (s,.){\Big |}_{L^{\infty}}
\max{j\in \left\lbrace 1,\cdots, n\right\rbrace }{\Big |}\delta v^{\rho,k,l}_j(s,.){\Big |}_{H^{1,\infty}}
\end{array}
\end{equation}
and for the second term
\begin{equation}
\begin{array}{ll}
{\Big|}\int_{{\mathbb R}^n}K_{n,m}(.-y)\frac{\partial}{\partial x_i}{\Big (} \left( \sum_{j,p=1}^n\left( v^{\rho,k,l}_{p,j}+v^{\rho,k-1,l}_{p,j}\right)(\tau,y) \right) \times\\
\\
\left( \delta v^{\rho,k,l}_{j,p}(\tau,y) \right) {\Big)}dy{\Big |}_{L^{\infty}}\\
\\
\leq {\Big|}\int_{{\mathbb R}^n}K_{n,m}(.-y)\frac{\partial}{\partial x_i}{\Big (} \left( \sum_{j,p=1}^n\left( v^{\rho,k,l}_{p,j}+v^{\rho,k-1,l}_{p,j}\right)(\tau,y) \right)
 {\Big)}dy{\Big |}_{L^{\infty}}\times\\
 \\
 {\Big |}\delta v^{\rho,k,l}_{j,p}(\tau,y) {\Big |}_{L^{\infty}}
\end{array}
\end{equation}
So we are left with the measure of a convolution in $L^{\infty}$ of the form
\begin{equation}
 {\Big|}\int_{{\mathbb R}^n}K_{n,m}(.-y)\frac{\partial}{\partial x_i}{\Big (} \left( \sum_{j,p=1}^n\left( v^{\rho,k,l}_{p,j}+v^{\rho,k-1,l}_{p,j}\right)(\tau,y) \right)
 {\Big)}dy{\Big |}_{L^{\infty}}
\end{equation}
Again this can be done by splitting up the kernel $\phi_{\epsilon}K_{n,m}(.)\in L^1$ and $(1-\phi_{\epsilon})K_{n,m}(.)\in L^2$, so the inductive assumptions that
\begin{equation}
{}\frac{\partial}{\partial x_i}\left( \sum_{j,p=1}^n\left( v^{\rho,k,l}_{p,j}+v^{\rho,k-1,l}_{p,j}\right)(\tau,.) \right)
 {\Big |}_{L^{\infty}}\leq C^l_k\leq C^l
\end{equation}
together with
\begin{equation}
{}\frac{\partial}{\partial x_i}\left( \sum_{j,p=1}^n\left( v^{\rho,k,l}_{p,j}+v^{\rho,k-1,l}_{p,j}\right)(\tau,.) \right)
 {\Big |}_{L^{2}}\leq C^l_k\leq C^l
\end{equation}
are indeed sufficient. 

\end{proof}

Although some parts of the argument (such as the local $L^1$-integrability of the first order derivatives of the Laplacian kernel) are valid only in dimension $n=3$, we note that we may extend this lemma by similar methods to Sobolev spaces for $n>3$ and to stronger Sobolev spaces. The extension to stronger Sobolev spaces is especially easy if we consider the $H^m$ estimates only at integer time points $l\geq 1$ which is sufficient if we add a little classical regularity theory for the appoximating functionals. We denote the extension for the controled scheme and prove it again in the case $n=3$.
\begin{lem}\label{lemmagen}
Let $n=3$ and let $v^{r,\rho,l-1}_i(l-1,.)\in H^m\cap H^{2,\infty}$ for $m\geq 2$
For $l\geq 1$ there is a time step size $\rho_l\sim \frac{1}{l}$ depending otherwise only on dimesnion and viscosity (explicit description in the proof) such that we have for all $\tau\in [l-1,l]$
\begin{equation}
v^{r,\rho,0,l}_i(\tau,.)\in H^m~\mbox{and}~\delta v^{r,\rho,k,l}(\tau,.)\in H^m
\end{equation}
and
\begin{equation}
\max_{i\in \left\lbrace 1,\cdots ,n\right\rbrace }\sup_{\tau\in [l-1,l]}|\delta v^{r,\rho,0,l}_i(\tau,.)|_{H^m}\leq \frac{1}{2},
\end{equation}
and for $k\geq 1$
\begin{equation}
\max_{i\in \left\lbrace 1,\cdots ,n\right\rbrace }\sup_{\tau\in [l-1,l]}|\delta v^{r,\rho,k,l}_i(\tau,.)|_{H^m}\leq \frac{1}{2}\max_{i\in \left\lbrace 1,\cdots ,n\right\rbrace }\sup_{\tau\in [l-1,l]}|\delta v^{r,\rho,k-1,l}_i(\tau,.)|_{H^m}
\end{equation}
for all $1\leq i\leq n$ and all $k\geq 0$. Furthermore similar estimates hold for $H^{2,\infty}$-spaces, i.e., we have
\begin{equation}
v^{r,\rho,0,l}_i(\tau,.)\in H^{2,\infty}~\mbox{and}~\delta v^{r,\rho,k,l}(\tau,.)\in H^{2,\infty}
\end{equation}
and
\begin{equation}
\max_{i\in \left\lbrace 1,\cdots ,n\right\rbrace }\sup_{\tau\in [l-1,l]}|\delta v^{r,\rho,0,l}_i(\tau,.)|_{H^{2,\infty}}\leq \frac{1}{2},
\end{equation}
and for $k\geq 1$
\begin{equation}
\max_{i\in \left\lbrace 1,\cdots ,n\right\rbrace }\sup_{\tau\in [l-1,l]}|\delta v^{r,\rho,k,l}_i(\tau,.)|_{H^{2,\infty}}\leq \frac{1}{2}\max_{i\in \left\lbrace 1,\cdots ,n\right\rbrace }\sup_{\tau\in [l-1,l]}|\delta v^{r,\rho,k-1,l}_i(\tau,.)|_{H^{2,\infty}}.
\end{equation}

\end{lem}
\begin{proof}
First we describe the controlled scheme. At each time step $l$ we have a functional series
\begin{equation}
\left( v^{r,\rho,k,l}_i\right)_{k\in {\mathbb N}},
\end{equation}
where for $k+1\geq 1$ the functions $v^{r,\rho,k,l}_i=v^{\rho,k,l}_i+r^l_i$ satisfy the Cauchy equation
\begin{equation}\label{Navlerayrkl}
\left\lbrace \begin{array}{ll}
\frac{\partial v^{r,\rho,k+1,l}_i}{\partial \tau}-\rho_l\nu\sum_{j=1}^n \frac{\partial^2 v^{r,\rho,k+1,l}_i}{\partial x_j^2} 
+\rho_l\sum_{j=1}^n v^{r,\rho,k,l}_j\frac{\partial v^{r,\rho,k+1,l}_i}{\partial x_j}=\\
\\
\frac{\partial r^l_i}{\partial \tau}-\rho_l\nu\sum_{j=1}^n \frac{\partial^2 r^l_i}{\partial x_j^2} 
+\rho_l\sum_{j=1}^n r^l_j\frac{\partial v^{r,\rho,k+1,l}_i}{\partial x_j}+\rho_l\sum_{j=1}^n v^{r,\rho,k+1,l}_j\frac{\partial r^l_i}{\partial x_j}\\
\\
+\rho_l\sum_{j=1}^n r^l_j\frac{\partial r^l_i}{\partial x_j}+\rho_l\int_{{\mathbb R}^n}\left( \frac{\partial}{\partial x_i}K_n(x-y)\right) \sum_{j,k=1}^n\left( \frac{\partial v^{r,\rho,k,l}_k}{\partial x_j}\frac{\partial v^{r,\rho,k,l}_j}{\partial x_k}\right) (\tau,y)dy,\\
\\
-2\rho_l\int_{{\mathbb R}^n}\left( \frac{\partial}{\partial x_i}K_n(x-y)\right) \sum_{j,k=1}^n\left( \frac{\partial v^{r,\rho,k,l}_k}{\partial x_j}\frac{\partial r^l_j}{\partial x_k}\right) (\tau,y)dy\\
\\
-\rho_l\int_{{\mathbb R}^n}\left( \frac{\partial}{\partial x_i}K_n(x-y)\right) \sum_{j,k=1}^n\left( \frac{\partial r^l_k}{\partial x_j}\frac{\partial r^l_j}{\partial x_k}\right) (\tau,y)dy,\\
\\
\mathbf{v}^{r,\rho,k+1,l}(l-1,.)=\mathbf{v}^{r,\rho,k,l}(l-1,.).
\end{array}\right.
\end{equation}

As described in the introduction, having computed $v^{r,\rho,l-1}_i(l-1,.)$ we first determine $v^{*,\rho,1,l}_i,~1\leq i\leq n$ via the linear equation 
\begin{equation}\label{Navlerayuncontrolledl}
\left\lbrace \begin{array}{ll}
\frac{\partial v^{*,\rho,1,l}_i}{\partial \tau}-\rho_l\nu\sum_{j=1}^n \frac{\partial^2 v^{r,\rho,1,l}_i}{\partial x_j^2} 
+\rho_l\sum_{j=1}^n v^{r,\rho,l-1}_j\frac{\partial v^{*,\rho,1,l}_i}{\partial x_j}=\\
\\+\rho_l\int_{{\mathbb R}^n}\left( \frac{\partial}{\partial x_i}K_n(x-y)\right) \sum_{j,k=1}^n\left( \frac{\partial v^{*,\rho,l-1}_k}{\partial x_j}\frac{\partial v^{\rho,l-1}_j}{\partial x_k}\right) (l-1,y)dy,\\
\\
\mathbf{v}^{*,\rho,1,l}(l-1,.)=\mathbf{v}^{r,\rho,l-1}(l-1,.).
\end{array}\right.
\end{equation}
Inductively, we have $v^{r,\rho,l-1}_i(l-1,.)\in C^2$, hence, classical theory of linear parabolic equations tells us that
\begin{equation}
v^{r,\rho,l}_i\in C^{1,2}\left([l-1,l]\times {\mathbb R}^n\right). 
\end{equation}
Then we define
\begin{equation}\label{controll}
r^l_i(.,.)=r^{l-1}_i-\left(v^{*,\rho,1,l}_i(.,.)-v^{r,\rho,l-1}_i(l-1,.)\right).
\end{equation}
Again, inductively we have $r^{l-1}_i(l-1,.)\in C^2$, hence we have 
\begin{equation}
r^l_i\in C^{1,2}\left([l-1,l]\times {\mathbb R}^n\right).
\end{equation}
With the control function $r^l_i$ defined in (\ref{controll}) we have to prove the contraction in $L^{\infty}\times H^2$ for the functional increments $\delta v^{r,\rho,k,l}_i,~1\leq i\leq n$, which satisfy the equation
\begin{equation}\label{Navleraycontrolledl}
\left\lbrace \begin{array}{ll}
\frac{\partial \delta v^{r,\rho,k,l}_i}{\partial \tau}-\rho_l\nu\sum_{j=1}^n \frac{\partial^2 \delta v^{r,\rho,k,l}_i}{\partial x_j^2} 
+\rho_l\sum_{j=1}^n v^{r,\rho,k-1,l}_j\frac{\partial \delta v^{r,\rho,k,l}_i}{\partial x_j}\\
\\
=-\rho_l\sum_{j=1}^n \delta v^{r,\rho,k-1,l}_j\frac{\partial  v^{r,\rho,k-1,l}_i}{\partial x_j}
+\rho_l\sum_{j=1}^n r^l_j\frac{\partial \delta v^{r,\rho,k,l}_i}{\partial x_j}
+\rho_l\sum_{j=1}^n \delta v^{r,\rho,k-1,l}_j\frac{\partial r^l_i}{\partial x_j}\\
\\+\rho_l\int_{{\mathbb R}^n}K_{n,i}(x-y){\Big (} \left( \sum_{j,k=1}^n\left( v^{\rho,k,l}_{k,j}+v^{\rho,k-1,l}_{k,j}\right)(\tau,y) \right)\left( \delta v^{\rho,k,l}_{j,k}(\tau,y) \right) {\Big)}dy,\\
\\
-2\rho_l\int_{{\mathbb R}^n}\left( \frac{\partial}{\partial x_i}K_n(x-y)\right) \sum_{j,k=1}^n\left( \frac{\partial \delta v^{r,\rho,k-1,l}_k}{\partial x_j}\frac{\partial r^l_j}{\partial x_k}\right) (\tau,y)dy\\
\\
\delta\mathbf{v}^{r,\rho,k,l}(l-1,.)=0.
\end{array}\right.
\end{equation}
From classical theory of scalar linear parabolic equations and Sobolev embedding we have for all $\tau\in [l-1,l]$ that
\begin{equation}\label{serk}
v^{r,\rho,k,l}_j(\tau,.)=v^{r,\rho,1,l}_j(\tau,.)+\sum_{m=2}^k\delta v^{r,\rho,m,l}(\tau,.)\in H^2\subset C^{\alpha}
\end{equation}
for $\alpha\in (0,0.5)$ and uniformly with respect to $\tau\in [l-1,l]$. Moreover, we know inductively that $v^{r,\rho,k,l}_j(\tau,.)\in C^2$ for all $\tau\in [l-1,l]$. For $k=1$ this is by definition of the control function. As in the case of the uncontolled scheme considered above inductively with respect to the subiteration index $k$ we know that the fundamental solution $p^{k,l}$ of
\begin{equation}
\frac{\partial p^{k,l}}{\partial \tau}-\rho_l \nu\sum_{j=1}^n \frac{\partial^2 \delta p^{k,l}}{\partial x_j^2} 
+\rho_l\sum_{j=1}^n v^{\rho,k,l}_j\frac{\partial p^{k,l}}{\partial x_j}=0
\end{equation}
exists, and it follows that the solution of the linear problem (\ref{Navleraycontrolledl}) has the representation
\begin{equation}\label{deltaurhok0proofr}
\begin{array}{ll}
 \delta v^{r,\rho,k+1,l}_i(\tau,x)=\\
\\
-\rho_l\int_{l-1}^{\tau}\int_{{\mathbb R}^n}\sum_j\left(\delta v^{r,\rho,k-1,l}_j\frac{\partial v^{r,\rho,k-1,l}_i}{\partial x_j} \right)(s,y)p^{k,l}(\tau,x,s,y)dyds\\ 
\\
+\rho_l\int_{l-1}^{\tau}\int_{{\mathbb R}^n}\int_{{\mathbb R}^n}K_{n,i}(z-y){\Big (} \left( \sum_{j,k=1}^n\left( v^{r,\rho,k,l}_{k,j}+v^{r,\rho,k-1,l}_{k,j}\right)(\tau,y) \right) \times\\
\\
\left( \delta v^{r,\rho,k,l}_{j,k}(\tau,y) \right) {\Big)}p^{k,l}(\tau,x,s,z)dydzds\\
\\
+\rho_l\int_{l-1}^{\tau}\int_{{\mathbb R}^n}\left( \sum_{j=1}^n r^l_j\frac{\partial \delta v^{r,\rho,k,l}_i}{\partial x_j}(s,y)\right) p^{k,l}(\tau,x,s,y)dyds\\
\\
+\rho_l\int_{l-1}^{\tau}\int_{{\mathbb R}^n}\left(\sum_{j=1}^n \delta v^{r,\rho,k-1,l}_j\frac{\partial r^l_i}{\partial x_j}\right)(s,y)p^{k,l}(\tau,x,s,z)dyds\\
\\
-2\rho_l\int_{l-1}^{\tau}\int_{{\mathbb R}^n}\int_{{\mathbb R}^n}\left( \frac{\partial}{\partial x_i}K_n(z-y)\right) \sum_{j,k=1}^n\left( \frac{\partial \delta v^{r,\rho,k-1,l}_k}{\partial x_j}\frac{\partial r^l_j}{\partial x_k}\right) (\tau,y)dy\times \\
\\
\times p^{k,l}(\tau,x,s,z)dzds.
\end{array}
\end{equation}
As in the case of the uncontrolled scheme we may differentiate under the integral such that we get the representation for  first order derivative $\delta v^{\rho,k+1,l}_{i,m}$ (Einstein notation) is obtained by replacing the fundamental solution $p^{k,l}$ in (\ref{deltaurhok0proofr}) by $p^{k,l}_{,m}$. 
Similarly we get the second order derivatives we have the representation by replacing first order derivatives $p^{k,l}_{,m}$ of the fundamental solution  $p^{k,l}$ by the adjoint$p^{k,l,*}_{,m}$ and adding one derivative at the other factors while shifting in the presence of Laplacian kernels. More precisely, we have 
\begin{equation}\label{deltaurhok0proof3r}
\begin{array}{ll}
 \frac{\partial^2}{\partial x_m\partial x_q}\delta v^{r,\rho,k+1,l}_i(\tau,x)=\\
\\
+\rho_l\int_{l-1}^{\tau}\int_{{\mathbb R}^n}\sum_j\frac{\partial}{\partial x_m}\left(\delta v^{r,\rho,k,l}_j\frac{\partial v^{r,\rho,k-1,l}}{\partial x_j} \right)(s,y)\frac{\partial}{\partial x_q}p^{k,l,*}(\tau,x,s,y)dyds\\ 
\\
+\rho_l\int_{l-1}^{\tau}\int_{{\mathbb R}^n}\int_{{\mathbb R}^n}K_{n,m}(z-y)\frac{\partial}{\partial x_i}{\Big (} \left( \sum_{j,p=1}^n\left( v^{r,\rho,k,l}_{p,j}+v^{r,\rho,k-1,l}_{p,j}\right)(\tau,y) \right) \times\\
\\
\left( \delta v^{r,\rho,k,l}_{j,p}(\tau,y) \right) {\Big)}\frac{\partial}{\partial x_q}p^{k,l,*}(\tau,x,s,z)dydsdydz\\
\\
+\rho_l\int_{l-1}^{\tau}\int_{{\mathbb R}^n}\left( \sum_{j=1}^n r^l_j\frac{\partial \delta v^{r,\rho,k,l}_i}{\partial x_j}(s,y)\right) p^{k,l,*}_{,q}(\tau,x,s,y)dyds\\
\\
+\rho_l\int_{l-1}^{\tau}\int_{{\mathbb R}^n}\left(\sum_{j=1}^n \delta v^{r,\rho,k-1,l}_j\frac{\partial r^l_i}{\partial x_j}\right)(s,y)p^{k,l,*}_{,q}(\tau,x,s,z)dyds\\
\\
-2\rho_l\int_{l-1}^{\tau}\int_{{\mathbb R}^n}\int_{{\mathbb R}^n}\left( \frac{\partial}{\partial x_i}K_n(z-y)\right) \sum_{j,k=1}^n\left( \frac{\partial \delta v^{r,\rho,k-1,l}_k}{\partial x_j}\frac{\partial r^l_j}{\partial x_k}\right) (s,y)dy\times \\
\\
\times p^{k,l,*}_{,q}(\tau,x,s,z)dzds.
\end{array}
\end{equation}
and where $p^{k,l,*}$ denotes the adjoint (consider also part I of this article). From these represenations we can prove contraction by the use of 
classical analysis ans some modest Sobolev theory. Let us first consider $H^1$ estimates.
Locally around the origin we may still use the standard a priori estimate for first order spatial derivatives of the Gaussian majorant of the adjoint of the fundamental so0lution, i.e.,
\begin{equation}
{\Big |} \frac{C}{(\tau-s)^{(n+1)/2}}\exp\left(-\frac{\lambda(.)^2}{4(\tau-s)} \right){\Big |}\leq \frac{c}{(t-s)^{\alpha}|x-y|^{n+1-2\alpha}}
\end{equation}
which holds for some constant $c>0$ and some parameter $\alpha\in (0.5,1)$. It is clear that a stronger local estimate holds for for the Gaussian itself.
On a domain which is  the complement of a ball around the origin the gaussian majorant of the adjoint of the fundamental solution behaves nicely. Hence, we surely have for all $\tau >s$
\begin{equation}\label{l1gauss}
\int_0^{\tau}{\Big|} \frac{C}{(\tau-s)^{(n+1)/2}}\exp\left(-\frac{\lambda(.)^2}{4(\tau-s)} \right){\Big |}_{L^1}ds\leq C
\end{equation}
for some constant $C>0$ wich is independent of $t-s$. First we observe that we may consider the spatial convolution first due to Fubini, and apply a Young inequality for fixed $t>s$. We get
\begin{equation}\label{deltaurhok0proof3r}
\begin{array}{ll}
 {\Big |}\frac{\partial^2}{\partial x_m\partial x_q}\delta v^{\rho,k+1,l}_i(\tau,.){\Big |}_{L^2}\\
\\
\leq \rho_l\int_{l-1}^{\tau}{\Big |}\sum_j\frac{\partial}{\partial x_m}\left(\delta v^{\rho,k,l}_j\frac{\partial v^{\rho,k-1,l}_i}{\partial x_j}\right)(s,.){\Big |}_{L^2}{\Big |} \frac{C}{(\tau-s)^{(n+1)/2}}\exp\left(-\frac{\lambda(.)^2}{4(\tau-s)} \right){\Big |}_{L^1}ds\\ 
\\
+\rho_l\int_{l-1}^{\tau}{\Big |}\int_{{\mathbb R}^n}K_{n,m}(.-y){\Big (} \frac{\partial}{\partial x_i}\left( \sum_{j,p=1}^n\left( v^{\rho,k,l}_{p,j}+v^{\rho,k-1,l}_{p,j}\right)(s,y) \right) \times\\
\\
\left( \delta v^{\rho,k,l}_{j,p}(s,y) \right) {\Big)}dy{\Big |}_{L^2}{\Big|} \frac{C}{(\tau-s)^{(n+1)/2}}\exp\left(-\frac{\lambda(.)^2}{4(\tau-s)} \right){\Big |}_{L^1}ds\\
\\
+\rho_l\int_{l-1}^{\tau}{\Big |} \sum_{j=1}^n r^l_j\frac{\partial \delta v^{r,\rho,k,l}_i}{\partial x_j}(s,.){\Big |}_{L^2} {\Big |} \frac{C}{(\tau-s)^{(n+1)/2}}\exp\left(-\frac{\lambda(.)^2}{4(\tau-s)} \right){\Big |}_{L^1}ds\\
\\
+\rho_l\int_{l-1}^{\tau}{\Big |}\left(\sum_{j=1}^n \delta v^{r,\rho,k-1,l}_j\frac{\partial r^l_i}{\partial x_j}\right)(s,.){\Big |}_{L^2} \frac{C}{(\tau-s)^{(n+1)/2}}\exp\left(-\frac{\lambda(.)^2}{4(\tau-s)} \right){\Big |}_{L^1}ds\\
\\
+2\rho_l\int_{l-1}^{\tau}{\Big |}\int_{{\mathbb R}^n}\left( \frac{\partial}{\partial x_i}K_n(.-y)\right) \sum_{j,k=1}^n\left( \frac{\partial \delta v^{r,\rho,k-1,l}_k}{\partial x_j}\frac{\partial r^l_j}{\partial x_k}\right) (s,.){\Big |}_{L^2}\times \\
\\
\times {\Big |} \frac{C}{(\tau-s)^{(n+1)/2}}\exp\left(-\frac{\lambda(.)^2}{4(\tau-s)} \right){\Big |}_{L^1}dzds.
\end{array}
\end{equation}
We still have convolution with respect to time and spatial convolutions involving first order derivatives of the Laplacian kernel. We postpone the consideration of the former and have a closer look at the latter first. As in the case of an uncontrolled scheme we consider partitions of unity 
\begin{equation}
K_{,i}=\phi_{\epsilon}K_{,i}+\left( 1-\phi_{\epsilon}\right)K_{,i},
\end{equation}
where $\phi_{\epsilon}$ is a locally supported smooth function which equals $1$ around the origin as it is defined above in our treatment of the uncontrolled scheme. This has the advantage that we can split up sums involving first order derivatives of the Laplacian kernel where we may use $\phi_{\epsilon}K_{,i}\in L^1$
and $\left( 1-\phi_{\epsilon}\right)K_{,i}\in L^2$. Let us consider the second term on the right side of (\ref{deltaurhok0proof3r}) first. Again it is the most convenient method in order to obatin a local $L^{\infty}\times H^2$ contraction result is to combine it with a $L^{\infty}\times H^{2,\infty}$ contraction result. This makes it easier to extract the function increment $\delta v^{r,\rho,l}_i$. Accordingly and inductively, we assume for every substage $k$ 
\begin{equation}\label{ckck}
\begin{array}{ll}
\max_{p\in \left\lbrace 1,\cdots ,n\right\rbrace }{\big |}\sup_{(s,y)\in [l-1,l]\times {\mathbb R}^n} v^{\rho,k,l}_{p}(s,y){\big |}+\\
\\
\max_{p,j\in \left\lbrace 1,\cdots ,n\right\rbrace }\sup_{(s,y)\in [l-1,l]\times {\mathbb R}^n}{\big |} v^{\rho,k,l}_{p,j}(s,y){\big |}+\\
\\
\max_{p,j,m\in \left\lbrace 1,\cdots ,n\right\rbrace }\sup_{(s,y)\in [l-1,l]\times {\mathbb R}^n}{\big |} v^{\rho,k,l}_{p,j,m}(s,y){\big |}
\leq C_k
\end{array}
\end{equation}
where $C_k>0$ is a nondercreasing sequence of constants which we want to have uniformal bounded.
We have 
\begin{equation}
\begin{array}{ll}
{\big |} v^{\rho,k,l}_{p,j}(s,y)+v^{\rho,k-1,l}_{p,j}(s,y){\big |}\\
\\
\leq {\big |} v^{\rho,k,l}_{p,j}(s,y){\big |}+{\big |}v^{\rho,k-1,l}_{p,j}(s,y){\big |}
\leq C_{k}+C_{k-1}\leq 2 C_k.
\end{array}
\end{equation}
Hence for the second term on the right side of (\ref{deltaurhok0proof3r}) we may write
\begin{equation}\label{laplaciankr}
\begin{array}{ll}
 \rho_l\int_{l-1}^{\tau}{\Big |}\int_{{\mathbb R}^n}K_{n,m}(.-y){\Big (} \frac{\partial}{\partial x_i}\left( \sum_{j,p=1}^n\left( v^{\rho,k,l}_{p,j}+v^{\rho,k-1,l}_{p,j}\right)(s,y) \right) 
\delta v^{\rho,k,l}_{j,p}(s,y){\Big)}dy{\Big |}_{L^2}ds\\
\\
\leq \rho_l\int_{l-1}^{\tau}n\max_{p\in \left\lbrace 1,\cdots ,n\right\rbrace }{\Big |}\int_{{\mathbb R}^n}K_{n,m}(.-y) 2n^2C_k
\delta v^{\rho,k,l}_{j,p}(s,y)  dy{\Big |}_{L^2}ds\\
\\
+\rho_l\int_{l-1}^{\tau}n\max_{p,q\in \left\lbrace 1,\cdots ,n\right\rbrace }{\Big |}\int_{{\mathbb R}^n}K_{n,m}(.-y) 2n^2C_k
\delta v^{\rho,k,l}_{j,p,q}(s,y)  dy{\Big |}_{L^2}ds\\
\\
\leq \rho_l\int_{l-1}^{\tau}n\max_{p\in \left\lbrace 1,\cdots ,n\right\rbrace }{\Big |}\int_{{\mathbb R}^n}\left( \phi_{\epsilon}K_{n,m}\right) (.-y) 2n^2C_k
\delta v^{\rho,k,l}_{j,p}(s,y)  dy{\Big |}_{L^2}ds\\
\\
+\rho_l\int_{l-1}^{\tau}n\max_{p,q\in \left\lbrace 1,\cdots ,n\right\rbrace }{\Big |}\int_{{\mathbb R}^n}\left( \phi_{\epsilon}K_{n,m}\right) (.-y) 2n^2C_k
\delta v^{\rho,k,l}_{j,p,q}(s,y)  dy{\Big |}_{L^2}ds\\
\\
+\rho_l\int_{l-1}^{\tau}n\max_{p\in \left\lbrace 1,\cdots ,n\right\rbrace }{\Big |}\int_{{\mathbb R}^n}\left( 1-\phi_{\epsilon}\right) K_{n,m}(.-y) 2n^2C_k
\delta v^{\rho,k,l}_{j,p}(s,y)  dy{\Big |}_{L^2}ds\\
\\
+\rho_l\int_{l-1}^{\tau}n\max_{p,q\in \left\lbrace 1,\cdots ,n\right\rbrace }{\Big |}\int_{{\mathbb R}^n}\left( 1-\phi_{\epsilon}\right) K_{n,m}(.-y) 2n^2C_k
\delta v^{\rho,k,l}_{j,p,q}(s,y)  dy{\Big |}_{L^2}ds
\end{array}
\end{equation}
The first two summands on the right side of (\ref{laplaciankr}) have a localised kernel in $L^1$ and can therefore be estimated by the Young inequality, i.e., with the bound
\begin{equation}
{\big |}\phi_{\epsilon}K_{,i}{\big |}_{L^1}\leq C_K
\end{equation}
we have
\begin{equation}\label{laplaciankr2}
\begin{array}{ll}
 \rho_l\int_{l-1}^{\tau}{\Big |}\int_{{\mathbb R}^n}K_{n,m}(.-y){\Big (} \frac{\partial}{\partial x_i}\left( \sum_{j,p=1}^n\left( v^{\rho,k,l}_{p,j}+v^{\rho,k-1,l}_{p,j}\right)(s,y) \right) 
\delta v^{\rho,k,l}_{j,p}(s,y){\Big)}dy{\Big |}_{L^2}ds\\
\\
\leq \rho_l\int_{l-1}^{\tau}2n^3C_kC_K\max_{p\in \left\lbrace 1,\cdots ,n\right\rbrace }{\Big |}
\delta v^{\rho,k,l}_{j,p}(s,.){\Big |}_{L^2}ds\\
\\
+\rho_l\int_{l-1}^{\tau}2n^3C_kC_K\max_{p,q\in \left\lbrace 1,\cdots ,n\right\rbrace }{\Big |}
\delta v^{\rho,k,l}_{j,p,q}(s,.) {\Big |}_{L^2}ds\\
\\
+\rho_l\int_{l-1}^{\tau}2n^3C_k\max_{p\in \left\lbrace 1,\cdots ,n\right\rbrace }{\Big |}\int_{{\mathbb R}^n}\left( 1-\phi_{\epsilon}\right) K_{n,m}(.-y)
\delta v^{\rho,k,l}_{j,p}(s,y)  dy{\Big |}_{L^2}ds\\
\\
+\rho_l\int_{l-1}^{\tau}2n^3C_k\max_{p,q\in \left\lbrace 1,\cdots ,n\right\rbrace }{\Big |}\int_{{\mathbb R}^n}\left( 1-\phi_{\epsilon}\right) K_{n,m}(.-y) 
\delta v^{\rho,k,l}_{j,p,q}(s,y)  dy{\Big |}_{L^2}ds
\end{array}
\end{equation}
In order to estimate the last to summands on the right side  we first observe that  
\begin{equation}
\left( 1-\phi_{\epsilon}\right) K_{n,m}(.) \in H^2,
\end{equation}
where we introduce some consatant $C^2_K$ such that
\begin{equation}
{\big |}\left( 1-\phi_{\epsilon}\right) K_{n,m}(.){\big |}_{ H^2}\leq C^2_K.
\end{equation}
We have
\begin{equation}\label{laplaciankr33}
\begin{array}{ll}
+\rho_l\int_{l-1}^{\tau}2n^3C_k\max_{p\in \left\lbrace 1,\cdots ,n\right\rbrace }{\Big |}\int_{{\mathbb R}^n}\left( 1-\phi_{\epsilon}\right) K_{n,m}(.-y)
\delta v^{\rho,k,l}_{j,p}(s,y)  dy{\Big |}_{L^2}ds\\
\\
+\rho_l\int_{l-1}^{\tau}2n^3C_k\max_{p,q\in \left\lbrace 1,\cdots ,n\right\rbrace }{\Big |}\int_{{\mathbb R}^n}\left( 1-\phi_{\epsilon}\right) K_{n,m}(.-y) 
\delta v^{\rho,k,l}_{j,p,q}(s,y)  dy{\Big |}_{L^2}ds\\
\\
\leq \rho_l\int_{l-1}^{\tau}2n^3C_k{\Big |}\int_{{\mathbb R}^n}\left( 1-\phi_{\epsilon}\right) K_{n,m}(.-y) 
\delta v^{\rho,k,l}_{j}(s,y)  dy{\Big |}_{H^2}ds
\end{array}
\end{equation}
Now in case of dimension $n=3$ for the term on the right side of (\ref{laplaciankr33}) we may
 apply the product rule for Sobolev spaces, i.e., the rule that
 \begin{equation}
 {\big |}fg{\big |}_{H^s}\leq C_s{\big |}f{\big |}_{H^s}{\big |}g{\big |}_{H^s}
 \end{equation}
for $s>\frac{n}{2}$ and some constant $C_s$. Hence for this $C_s$ with $s=2$ the right side of (\ref{laplaciankr33}) is bounded by
\begin{equation}
\rho_l\int_{l-1}^{\tau}2n^3C_kC^2_KC_s{\Big |}
\delta v^{\rho,k,l}_{j}(s,.){\Big |}_{H^2}ds.
\end{equation}
Summing up the argument we have
\begin{equation}\label{laplaciankr2223}
\begin{array}{ll}
 \rho_l\int_{l-1}^{\tau}{\Big |}\int_{{\mathbb R}^n}K_{n,m}(.-y){\Big (} \frac{\partial}{\partial x_i}\left( \sum_{j,p=1}^n\left( v^{\rho,k,l}_{p,j}+v^{\rho,k-1,l}_{p,j}\right)(s,y) \right) 
\delta v^{\rho,k,l}_{j,p}(s,y){\Big)}dy{\Big |}_{L^2}ds\\
\\
\leq \rho_l\int_{l-1}^{\tau}2n^3C_kC_K\max_{p\in \left\lbrace 1,\cdots ,n\right\rbrace }{\Big |}
\delta v^{\rho,k,l}_{j,p}(s,.){\Big |}_{L^2}ds\\
\\
+\rho_l\int_{l-1}^{\tau}2n^3C_kC_K\max_{p,q\in \left\lbrace 1,\cdots ,n\right\rbrace }{\Big |}
\delta v^{\rho,k,l}_{j,p,q}(s,.) {\Big |}_{L^2}ds\\
\\
+\rho_l\int_{l-1}^{\tau}2n^3C_kC^2_KC_s{\Big |}
\delta v^{\rho,k,l}_{j}(s,.){\Big |}_{H^2}ds.
\end{array}
\end{equation} 
Lets go back to (\ref{deltaurhok0proof3r}). There are two convolutions with first order spatial derivatives of the Laplacian kernel. One is estimated above, and the other (the last bterm on the right side of (\ref{deltaurhok0proof3r}) involving the control function $r^l_i$ can be estimated similarly if we introduce the constant
\begin{equation}\label{crcr}
\begin{array}{ll}
\max_{i\in\left\lbrace 1,\cdots ,n\right\rbrace }\sup_{\tau\in [l-1,l],y\in {\mathbb R}^n}{\big |}r^l_i(\tau,y){\big |}+\\
\\
\max_{i,p\in\left\lbrace 1,\cdots ,n\right\rbrace }\sup_{\tau\in [l-1,l],y\in {\mathbb R}^n}{\big |}r^l_{i,p}(\tau,y){\big |}\\
\\
+\max_{i,p,q\in\left\lbrace 1,\cdots ,n\right\rbrace }\sup_{\tau\in [l-1,l],y\in {\mathbb R}^n}{\big |}r^l_{i,p,q}(\tau,y){\big |}\leq C_r.
\end{array}
\end{equation}
We apply these estimates to the right side of (\ref{deltaurhok0proof3r}) and get

\begin{equation}\label{deltaurhok0proof3rr}
\begin{array}{ll}
 {\Big |}\frac{\partial^2}{\partial x_m\partial x_q}\delta v^{\rho,k+1,l}_i(\tau,.){\Big |}_{L^2}\\
\\
\leq \rho_l\int_{l-1}^{\tau}{\Big |}\sum_j\frac{\partial}{\partial x_m}\left(\delta v^{\rho,k,l}_j\frac{\partial v^{\rho,k-1,l}_i}{\partial x_j}\right)(s,.){\Big |}_{L^2}{\Big |} \frac{C}{(\tau-s)^{(n+1)/2}}\exp\left(-\frac{\lambda(.)^2}{4(\tau-s)} \right){\Big |}_{L^1}ds\\ 
\\
+\rho_l\int_{l-1}^{\tau}2n^3C_kC^2_KC_s{\Big |}
\delta v^{\rho,k,l}_{j}(s,.){\Big |}_{H^2}{\Big|} \frac{C}{(\tau-s)^{(n+1)/2}}\exp\left(-\frac{\lambda(.)^2}{4(\tau-s)} \right){\Big |}_{L^1}ds\\
\\
+\rho_l\int_{l-1}^{\tau}{\Big |} \sum_{j=1}^n r^l_j\frac{\partial \delta v^{r,\rho,k,l}_i}{\partial x_j}(s,.){\Big |}_{L^2} {\Big |} \frac{C}{(\tau-s)^{(n+1)/2}}\exp\left(-\frac{\lambda(.)^2}{4(\tau-s)} \right){\Big |}_{L^1}ds\\
\\
+\rho_l\int_{l-1}^{\tau}{\Big |}\left(\sum_{j=1}^n \delta v^{r,\rho,k-1,l}_j\frac{\partial r^l_i}{\partial x_j}\right)(s,.){\Big |}_{L^2} \frac{C}{(\tau-s)^{(n+1)/2}}\exp\left(-\frac{\lambda(.)^2}{4(\tau-s)} \right){\Big |}_{L^1}ds\\
\\
2\rho_l\int_{l-1}^{\tau}2n^3C_rC^2_KC_s{\Big |}
\delta v^{\rho,k,l}_{j}(s,.){\Big |}_{H^2}{\Big|} \frac{C}{(\tau-s)^{(n+1)/2}}\exp\left(-\frac{\lambda(.)^2}{4(\tau-s)} \right){\Big |}_{L^1}ds.
\end{array}
\end{equation}
Next we use the upper bounds in (\ref{ckck}) and (\ref{crcr}) in order get
\begin{equation}\label{deltaurhok0proof3rr}
\begin{array}{ll}
 {\Big |}\frac{\partial^2}{\partial x_m\partial x_q}\delta v^{\rho,k+1,l}_i(\tau,.){\Big |}_{L^2}\\
\\
\leq \sum_{j=1}^n\rho_l\int_{l-1}^{\tau}{\Big |}nC_k\left(\delta v^{\rho,k,l}_j\right)(s,.){\Big |}_{H^2}{\Big |} \frac{C}{(\tau-s)^{(n+1)/2}}\exp\left(-\frac{\lambda(.)^2}{4(\tau-s)} \right){\Big |}_{L^1}ds\\ 
\\
+\rho_l\int_{l-1}^{\tau}2n^3C_kC^2_KC_s{\Big |}
\delta v^{\rho,k,l}_{j}(s,.){\Big |}_{H^2}{\Big|} \frac{C}{(\tau-s)^{(n+1)/2}}\exp\left(-\frac{\lambda(.)^2}{4(\tau-s)} \right){\Big |}_{L^1}ds\\
\\
+\sum_{i,j=1}^n \rho_l\int_{l-1}^{\tau}{\Big |} nC_r\frac{\partial \delta v^{r,\rho,k,l}_i}{\partial x_j}(s,.){\Big |}_{H^1} {\Big |} \frac{C}{(\tau-s)^{(n+1)/2}}\exp\left(-\frac{\lambda(.)^2}{4(\tau-s)} \right){\Big |}_{L^1}ds\\
\\
+\sum_{j=1}^n\rho_l\int_{l-1}^{\tau}{\Big |}nC_r\delta v^{r,\rho,k-1,l}_j(s,.){\Big |}_{H^1} \frac{C}{(\tau-s)^{(n+1)/2}}\exp\left(-\frac{\lambda(.)^2}{4(\tau-s)} \right){\Big |}_{L^1}ds\\
\\
+\sum_{j=1}^n2\rho_l\int_{l-1}^{\tau}2n^3C_rC^2_KC_s{\Big |}
\delta v^{\rho,k,l}_{j}(s,.){\Big |}_{H^2}{\Big|} \frac{C}{(\tau-s)^{(n+1)/2}}\exp\left(-\frac{\lambda(.)^2}{4(\tau-s)} \right){\Big |}_{L^1}ds.
\end{array}
\end{equation}
We have a lot of convolutions with respect to time in the latter expression. However, since the factors are positive for each summand we may take the supremum of one factor and integrate the other one over time using the Gaussian estimate above. Inductively, all the functions (integrands in (\ref{deltaurhok0proof3rr}))
 \begin{equation}
 \begin{array}{ll}
s\rightarrow {\Big |}nC_k\left(\delta v^{\rho,k,l}_j\right)(s,.){\Big |}_{H^2}\\
 s\rightarrow {\Big |} nC_r\frac{\partial \delta v^{r,\rho,k,l}_i}{\partial x_j}(s,.){\Big |}_{H^1}
 \end{array}
 \end{equation}
are in $L^{\infty}\left(\left[l-1,l\right]\right)\cap C\left(\left[l-1,l\right]\right)$ (especially, continuous and bounded), hence we may estimate by suprema writing
\begin{equation}\label{deltaurhok0proof3rr}
\begin{array}{ll}
 {\Big |}\frac{\partial^2}{\partial x_m\partial x_q}\delta v^{\rho,k+1,l}_i(\tau,.){\Big |}_{L^2}\\
\\
\leq \sum_{j=1}^n\rho_l\sup_{s\in [l-1,l]}{\Big |}nC_k\left(\delta v^{\rho,k,l}_j\right)(s,.){\Big |}_{H^2}\int_{l-1}^{\tau}{\Big |} \frac{C}{(\tau-s)^{(n+1)/2}}\exp\left(-\frac{\lambda(.)^2}{4(\tau-s)} \right){\Big |}_{L^1}ds\\ 
\\
+\sum_{j=1}^n\rho_l2n^3C_kC^2_KC_s\sup_{s\in [l-1,l]}{\Big |}
\delta v^{\rho,k,l}_{j}(s,.){\Big |}_{H^2}\int_{l-1}^{\tau}{\Big|} \frac{C}{(\tau-s)^{(n+1)/2}}\exp\left(-\frac{\lambda(.)^2}{4(\tau-s)} \right){\Big |}_{L^1}ds\\
\\
+\sum_{i,j=1}^n \rho_l\sup_{s\in [l-1,l]}{\Big |} nC_r\frac{\partial \delta v^{r,\rho,k,l}_i}{\partial x_j}(s,.){\Big |}_{H^1} \int_{l-1}^{\tau}{\Big |} \frac{C}{(\tau-s)^{(n+1)/2}}\exp\left(-\frac{\lambda(.)^2}{4(\tau-s)} \right){\Big |}_{L^1}ds\\
\\
+\sum_{j=1}^n\rho_l\sup_{s\in [l-1,l]}{\Big |}nC_r\delta v^{r,\rho,k-1,l}_j(s,.){\Big |}_{H^1}\int_{l-1}^{\tau}{\Big |} \frac{C}{(\tau-s)^{(n+1)/2}}\exp\left(-\frac{\lambda(.)^2}{4(\tau-s)} \right){\Big |}_{L^1}ds\\
\\
+\sum_{j=1}^n2\rho_l\sup_{s\in [l-1,l]}2n^3C_rC^2_KC_s{\Big |}
\delta v^{\rho,k,l}_{j}(s,.){\Big |}_{H^2}ds\int_{l-1}^{\tau}\int_{l-1}^{\tau}{\Big|} \frac{C}{(\tau-s)^{(n+1)/2}}\exp\left(-\frac{\lambda(.)^2}{4(\tau-s)} \right){\Big |}_{L^1}ds.
\end{array}
\end{equation}

We assume (generic) $C>0$ is the $L^1$-bound of the time-integrated spatial derivative of the Gaussian as in (\ref{l1gauss}). We get
\begin{equation}\label{deltaurhok0proof3rr}
\begin{array}{ll}
 {\Big |}\frac{\partial^2}{\partial x_m\partial x_q}\delta v^{\rho,k+1,l}_i(\tau,.){\Big |}_{L^2}\\
\\
\leq \sum_{j=1}^n\rho_l\sup_{s\in [l-1,l]}{\Big |}nC_k\left(\delta v^{\rho,k,l}_j\right)(s,.){\Big |}_{H^2}C\\ 
\\
+\sum_{j=1}^n\rho_l2n^3C^l_kC^2_KC_s\sup_{s\in [l-1,l]}{\Big |}
\delta v^{\rho,k,l}_{j}(s,.){\Big |}_{H^2}C\\
\\
+\sum_{i,j=1}^n \rho_l\sup_{s\in [l-1,l]}{\Big |} nC^l_r\frac{\partial \delta v^{r,\rho,k,l}_i}{\partial x_j}(s,.){\Big |}_{H^1} C\\
\\
+\sum_{j=1}^n\rho_l\sup_{s\in [l-1,l]}{\Big |}nC^l_r\delta v^{r,\rho,k-1,l}_j(s,.){\Big |}_{H^1}C\\
\\
+\sum_{j=1}^n2\rho_l\sup_{s\in [l-1,l]}2n^3C^l_rC^2_KC_s{\Big |}
\delta v^{\rho,k,l}_{j}(s,.){\Big |}_{H^2}dsC.
\end{array}
\end{equation}
We can estimate the right side withe repect to one $H^2$-norm (with supremum over time). Summing up constants and assuming w.l.o.g. that $C^l_k,C^l_r,C^2_K,C_s,C\geq 1$ we have
\begin{equation}\label{deltaurhok0proof3rrconst*}
\begin{array}{ll}
 {\Big |}\frac{\partial^2}{\partial x_m\partial x_q}\delta v^{\rho,k+1,l}_i(\tau,.){\Big |}_{L^2}\\
\\
\leq \rho_l(nC^l_kC+2n^4C^l_kC^2_KC_s+n^3C^l_r+n^2C_r+4n^3C^l_rC^2_KC_s)\times\\
\\
\times \sup_{s\in [l-1,l]}{\Big |}\delta v^{\rho,k,l}_j(s,.){\Big |}_{H^2}\\
\\
\leq \rho_l4n^4(C^l_k+C^l_r)C^2_KC_sC\sup_{s\in [l-1,l]}{\Big |}\delta v^{\rho,k,l}_j(s,.){\Big |}_{H^2}
\end{array}
\end{equation}
The inductive linear growth of the constants $C^l_k+C^l_r$ with respect to the time step number $l\geq 1$ is consumated by the time step size $\rho_l$. At his point we still have a series of constants $C^l_k$ which depend on the substep number $k$. However, using analogous observations as in the uncontrolled case above we can repeat the argument above for $H^{2,\infty}$ spaces.
We have
\begin{equation}
\max_{i\in \left\lbrace 1,\cdots ,n\right\rbrace }\sup_{\tau\in [l-1,l]}|\delta v^{r,\rho,0,l}_i(\tau,.)|_{H^{2,\infty}}\leq \frac{1}{2},
\end{equation}
and for $k\geq 1$
\begin{equation}
\max_{i\in \left\lbrace 1,\cdots ,n\right\rbrace }\sup_{\tau\in [l-1,l]}|\delta v^{r,\rho,k,l}_i(\tau,.)|_{H^{2,\infty}}\leq \frac{1}{2}\max_{i\in \left\lbrace 1,\cdots ,n\right\rbrace }\sup_{\tau\in [l-1,l]}|\delta v^{r,\rho,k-1,l}_i(\tau,.)|_{H^{2,\infty}}.
\end{equation}
for $\rho_l$ as in (\ref{rhol**}) below.
This leads to the conclusion that we may define
\begin{equation}
C^l=2C^l_1,
\end{equation}
where
\begin{equation}
C^l_k\leq C^l \mbox{ for all }k\geq 1.
\end{equation}
This justifies the induction hypothesis concerning the $H^{2,\infty}$-norm of the approximations $v^{r,\rho,k,l}_i$, and we may choose 
\begin{equation}\label{rhol**}
\rho_l=\frac{1}{2(n+1)^2\rho_l4n^4(C^l+C^l_r)C^2_KC_sC}
\end{equation}
The additional factor in the denominator (apart from $2$) is due to the fact that we have other terms in the $H^2$-norm of course. Indeed we may count $1+n+n^2$ in the classical definition of the $H^2$ norm. Then we observe that the first derivatives can be estimated by the same argument with the first order derivatives of the fundamental solution replaced by the fundamental solution itself, and the first order derivative of the Gaussian a priori majorant replaced by the Gaussian itself. The estimate even simplifies since we do not need the adjoint. Clearly, $L^2$ estimates can be achieved similarly and have the right side (\ref{deltaurhok0proof3rrconst*}) as an upper bound a fortiori. We have
\begin{equation}\label{deltaurhok0proof3rrconst*der1}
\begin{array}{ll}
 {\Big |}\frac{\partial}{\partial x_q}\delta v^{\rho,k+1,l}_i(\tau,.){\Big |}_{L^2}\\
\\
\leq \rho_l(nC^l_kC+2n^4C^lC^2_KC_s+n^3C^l_r+n^2C_r+4n^3C^l_rC^2_KC_s)\times\\
\\
\times \sup_{s\in [l-1,l]}{\Big |}\delta v^{\rho,k,l}_j(s,.){\Big |}_{H^2}\\
\\
\leq \rho_l4n^4(C^l+C^l_r)C^2_KC_sC\sup_{s\in [l-1,l]}{\Big |}\delta v^{\rho,k,l}_j(s,.){\Big |}_{H^2},
\end{array}
\end{equation}
and
\begin{equation}\label{deltaurhok0proof3rrconst*der1}
\begin{array}{ll}
 {\Big |}\delta v^{\rho,k+1,l}_i(\tau,.){\Big |}_{L^2}\\
\\
\leq \rho_l4n^4(C^l_k+C^l_r)C^2_KC_sC\sup_{s\in [l-1,l]}{\Big |}\delta v^{\rho,k,l}_j(s,.){\Big |}_{H^2},
\end{array}
\end{equation}
and with the choice of $\rho_l$ in (\ref{rhol}) we have
\begin{equation}\label{deltaurhok0proof3rrconst*der1}
\begin{array}{ll}
 \sup_{\tau\in [l-1,l]}{\Big |}\delta v^{\rho,k+1,l}_i(\tau,.){\Big |}_{H^2}\\
\\
\leq \frac{1}{2}\sup_{s\in [l-1,l]}{\Big |}\delta v^{\rho,k,l}_j(s,.){\Big |}_{H^2},
\end{array}
\end{equation}
as desired.

\end{proof}

Recall that in the controlled scheme we choose
\begin{equation}
\delta r^l_i=r^l_i-r^l_i(l-1,.):=- \left( v^{*,\rho,1,l}_i-v^{r,\rho,l-1}_i(l-1,.)\right),
\end{equation}
where $v^{r,\rho,1,l}_i:=v^{*,\rho,1,l}_i$.
Next we observe that for this choice we have indeed linear growth on a time scale $\rho_l\sim \frac{1}{l}$, or that the controlled scheme is global. We have
\begin{thm}
For all $m\geq 0$ there is a a constant $C_m$ independent of the time step number $l$ such that
\begin{equation}
|v^{r,\rho,l}_i(l,.)|_{H^m}\leq |v^{r,\rho,l-1}_i(l-1,.)|_{H^m}+C_m
\end{equation}
\end{thm}
\begin{proof}
For $m=2$ the proof of this lemma is almost included in the proof of lemma \ref{lemmagen}, which is a a natural extension of the lemma \ref{contrlem} above. First the choice
\begin{equation}
r^l_i=r^{l-1}_i(l-1,.)-\left( v^{*,\rho,1,l}_i-v^{r,\rho,l-1}_i(l-1,.)\right)
\end{equation}
implies
\begin{equation}
\begin{array}{ll}
v^{r,\rho,1,l}_i=v^{*,\rho,1,l}_i+\delta r^l_i=v^{r,\rho,l-1}_i(l-1,.).
\end{array}
\end{equation}
Hence
\begin{equation}
\begin{array}{ll}
v^{r,\rho,l}_i=v^{r,\rho,1,l}_i+\sum_{k=2}^{\infty}\delta v^{r,\rho,k,l}_i\\
\\
=v^{r,\rho,l-1}_i(l-1,.)+\sum_{k=2}^{\infty}\delta v^{r,\rho,k,l}_i
\end{array}
\end{equation}
From the preceding lemma we have
\begin{equation}
\begin{array}{ll}
|v^{r,\rho,l}_i(l,.)|_{H^m}=|v^{r,\rho,1,l}_i(l,.)|_{H^m}+\sum_{k=2}^{\infty}|\delta v^{r,\rho,k,l}_i(l,.)|_{H^m}\\
\\
\leq |v^{r,\rho,l-1}_i(l-1,.)|_{H^m}+2|\delta v^{r,\rho,1,l}_i(l,.)|_{H^m}\\
\\
\leq |v^{r,\rho,l-1}_i(l-1,.)|_{H^m}+1.
\end{array}
\end{equation}
We conclude that we have linear growth on a time scale defined by the time step size $\rho_l\sim \frac{1}{l}$. This implies that the first order coefficients of the linearized equations for $v^{\rho,0,l}_i,1\leq i\leq n$ are uniformly bounded, i.e., we have for any H\"{o}lder norm $|.|_{\alpha}$ with H\"{o}lder coefficient $\alpha\in (0,1)$
\begin{equation}
|v^{\rho,l-1}(l-1,.)|_{\alpha}\leq C
\end{equation}
for some constant $C>0$ independently of the time step number $l$. Moreover, we have H\"{o}lder continuity of all correction terms $\delta v^{\rho,k,l}_i,~1\leq i\leq n,~k\geq 1$ uniformly in $\tau\in [l-1,l]$ and independently of the time step number $l\geq 1$. This implies that the scheme is global.
\end{proof}
The argument we have proposed here is considerable simpler then the argument in \cite{K3} where a dynamic control function used is much more complicated. On the other hand, the dynamic control function used in \cite{K3} implies directly that the function is globally bounded and it may stabilize the scheme. Furthermore it allows for a uniform time step size which is an advantage from the numerical point of view. It may also be used in other situations.  Note that the present argument also leads to a different proof of the classical proofs for global $L^2$-existence in \cite{H} and \cite{L}. It would be interesting to apply the scheme using probabilistic methods considered in \cite{FrKa}, \cite{FK}, \cite{K}, and \cite{KKS}. 

 \footnotetext[1]{
\texttt{{kampen@wias-berlin.de, kampen@mathalgorithm.de}}.}


\begin{thebibliography}{19}
\baselineskip=12pt






\bibitem{FrKa}
{\sc Fries, Christian; Kampen, J\"org}: 
{Proxy Simulation Schemes for generic robust Monte Carlo sensitivities, process oriented importance sampling and high accuracy drift approximation (with applications to the LIBOR market model)}, Journal of Computational Finance, Vol. 10, Nr. 2, 97-128, 2007.

\bibitem{FK}
{\sc Fries, C., Kampen, J.} {\em Global regularity, existence and a probabilistic scheme for a class of ultraparabolic equations} (in preperation)






 \bibitem{H}
 {\sc Hopf, H.} {\em  \"Uber die Anfangswertaufgabef\"ur die hydrodynamischen Grundgleichungen}, Math. Nachr.4, 213-231,1951.

\bibitem{KKS}
{\sc Kampen, J., Kolodko, A., Schoenmakers, J.}, {\em Monte Carlo Greeks for financial products via approximative transition densities}, Siam J. Sc. Comp., vol. 31 , p. 1-22, 2008.

\bibitem{K}
{\sc Kampen, J.,} {\em Global regularity and probabilistic schemes for free boundary surfaces of multivariate American derivatives and their Greeks},  Siam J. Appl. Math. 71, pp. 288-308.

%
%
 \bibitem{K3}
 {\sc Kampen, J.,} {\em A global scheme for the incompressible Navier-Stokes equation on compact Riemannian manifolds}, arXiv: 1205.4888v4,  (June 2012)

 \bibitem{KNS}
	{\sc Kampen, J\"org}: {\em Constructive analysis of the Navier-Stokes equation.} arXiv10044589 (v6), 2012
 \bibitem{KB1}
{\sc Kampen, J.} {\em On the multivariate Burgers equation and the incompressible Navier-Stokes equation (part I)}, arXiv:0910.5672v5  [math.AP] (2011)


 \bibitem{KB3}
{\sc Kampen, J.} {\em On the multivariate Burgers equation and the incompressible Navier-Stokes equation (part III)}, arXiv:v2  [math.AP] December (2012)

	
	
\bibitem{L}
{\sc Leray, J.} {\em Sur le Mouvement d'un Liquide Visquex Emplissent l'Espace}, Acta Math. J. (63), 193-248, (1934).


















 \end{thebibliography}
\end{document}